\documentclass[11pt, oneside]{amsart}  	
\usepackage{geometry}               
\usepackage{graphicx}				
\usepackage{amsthm,amsmath,amssymb,graphicx,mathrsfs,braket,latexsym}

\newtheorem{thm}{Theorem}[section]
\newtheorem{prop}[thm]{Proposition}
\newtheorem{lem}[thm]{Lemma}
\newtheorem{cor}[thm]{Corollary}
\theoremstyle{definition}
\newtheorem{defi}[thm]{Definition}
\newtheorem{rem}[thm]{Remark}
\newtheorem{ex}[thm]{Example}
\newtheorem{assum}[thm]{Assumption}
\newtheorem{nota}[thm]{Notation}

\newtheorem{setting}{Setting}
\newtheorem{problem}{Problem}

\DeclareMathOperator{\supp}{supp}
\DeclareMathOperator{\card}{card}
\DeclareMathOperator{\diam}{diam}

\DeclareMathOperator{\Patch}{Patch}
\DeclareMathOperator{\Pattern}{Pattern}

\DeclareMathOperator{\LF}{LF}
\DeclareMathOperator{\Map}{Map}
\DeclareMathOperator{\UD}{UD}
\DeclareMathOperator{\Sym}{Sym}
\DeclareMathOperator{\Aut}{Aut}
\DeclareMathOperator{\Del}{Del}
\DeclareMathOperator{\calC}{Cl}

\newcommand{\Od}{\mbox{O}(d)}
\newcommand{\Ed}{\mbox{E}(d)}

\newcommand{\calT}{\mathcal{T}}

\newcommand{\calS}{\mathcal{S}}
\newcommand{\calP}{\mathcal{P}}
\newcommand{\calQ}{\mathcal{Q}}
\newcommand{\calR}{\mathcal{R}}
\newcommand{\calE}{\mathcal{E}}
\newcommand{\frakF}{\mathfrak{F}}
\newcommand{\e}{\varepsilon}
\newcommand{\Rd}{\mathbb{R}^d}

\newcommand{\LD}{\overset{\mathrm{LD}}{\rightarrow}}
\newcommand{\MLD}{\overset{\mathrm{MLD}}{\leftrightarrow}}

\newcommand{\sci}{\wedge}

\renewcommand{\labelenumi}{\arabic{enumi}.}

\title[A general framework]{A general framework for tilings, Delone sets, functions and measures, and
their interrelation}
\author{Yasushi Nagai}
\address{Montanuniversit\"at, Department Mathematik und Informationstechnologie,
Lehrstuhl f\"ur Mathematik und Statistik,
Franz Josef Strasse 18, A-8700 Leoben, Austria}
\email{yasushi.nagai@unileoben.ac.at}
\date{\today}					
\thanks{The author was supported by the project I3346 of the Japan Society for the Promotion of Science (JSPS) and the Austrian Science Fund (FWF)}


\begin{document}

\maketitle

\begin{abstract}
        We define a general framework that includes objects such as tilings, Delone sets, 
        functions and measures.  We define local derivability and mutual local
        derivability (MLD) between any two of these objects in order to describe their
         interrelation. This is a generalization
  of the local derivability and MLD (or S-MLD) for tilings and
       Delone sets
        which are used in the literature, under a mild assumption.
       We show that
 several canonical maps in aperiodic order send an object $\calP$ to one
      that is MLD with $\calP$.
      Moreover we show that,
       for an object $\calP$ and a class $\Sigma$ of objects,
       a mild condition on them assures that
      there exists some $\calQ\in\Sigma$ that is MLD with $\calP$.
      As an application, we study pattern equivariant functions.
      In particular, we show that the space of all pattern-equivariant functions
     contains all the information of the original object up to MLD in a quite general
 setting.
\end{abstract}

\section{Introduction}

Objects such as tilings, Delone (multi) sets, measures and almost periodic functions
have been investigated in the literature, especially after the discovery of
quasicrystals in materials science.
Quasicrystals are not periodic but have long-range order, and the above mathematical
objects with similar properties are studied intensively. Especially, non-periodic
objects with pure point diffraction measures are interesting.
It is a fundamental problem to study which objects have pure point diffraction measures.
A classification of such objects is an ultimate goal.

To define the diffraction measure, one has to convert objects such as tilings to
measures. There are standard ways of converting, such as putting Dirac measures to
 each of points in a Delone set. It is also useful to convert Delone sets to
 tilings, since for certain tilings one has a theory of deformation \cite{MR1971208}
 \cite{MR2201938} and
 cohomology \cite{MR2371606}.
 It is also useful to convert certain tilings (such as Penrose tilings) to
 Meyer sets, a special case of Delone sets for which we have the equivalence of
 algebraic and analytical definitions (\cite{MR1460032}).
In all these cases the original object and the converted one are considered to be
mutually locally
derivable (MLD), which was defined in \cite{MR1132337}.
Two objects that are MLD are considered to be essentially same, at least under the
assumption of finite local complexity.

 However, often these techniques of conversions are folklore. In particular, MLD is
 defined only for patterns. In this article, we generalize the definition of MLD
 (and S-MLD \cite{MR1132337}, if we consider $\Od$-actions) to
 include other objects and
 show under standard conversions the original one and the converted one are MLD.
 Actually we show MLD under a more general setting than is known. 
 This becomes a reference for many researches using such conversions.

This is done by constructing a general framework, \textit{abstract pattern space},
 that includes the above objects
such as tilings and Delone sets.  Each example of abstract pattern space contains
the objects of interest, such as tilings and
Delone sets. In general, we call these objects of interest
\textit{abstract patterns}.
The framework of abstract pattern space is enough to define local derivability and MLD
between abstract patterns.

The framework is very general so that it includes many non-interesting examples.
We often restrain ourselves to interesting cases by putting the following assumptions.
\begin{assum}\label{simplicity_assumption}
\begin{enumerate}
 \item The diameters of components of each abstract pattern are bounded from above.
       For patterns, this means the diameters of elements of each pattern is bounded.
       For other abstract patterns, we assume they
       consists of bounded components
       (Definition \ref{def_bounded_coponents}).
 \item The objects are Delone-deriving (Definition \ref{def_Delone_deriving}).
       That means a Delone set is locally derivable form each object.
 \item Each set $\Sigma$ of abstract patterns of interest is assumed to be
       supremum-closed 
       and  inside a glueable abstract pattern space
       (Definition \ref{def_glueable_pattern_space}).
       This means we can take the ``union'' of ``nice'' family of objects.
       For example, the abstract pattern space of patches satisfies this property since
       we can take the union $\bigcup_i\calP_i$ to obtain a new patch from a family
       of patches $\{\calP_i\mid i\in I\}$ such that tiles in these patches either
       do not intersect or coincide.
\end{enumerate} 
\end{assum}
Under these assumptions our theory of MLD is rich enough to include many examples and
simple enough so that we can prove various results, including the equivalence of our
definition and the one in the literature.

Let us explain the plan of this article more concretely.
Objects such as tilings and Delone sets admit the following
structures, which play important roles explicitly or implicitly.
\begin{enumerate}
 \item They admit cutting-off operation. For example, if $\calT$ is a tiling in $\Rd$ and
       $C\subset\Rd$, we can ``cut off'' $\calT$ by $C$ by considering
       \begin{align*}
	    \calT\sci C=\{T\in\calT\mid T\subset C\}.
       \end{align*}
       By this operation we forget the behavior of $\calT$ outside $C$.
\item Some of the objects ``include'' other objects. For patches this means the usual
      inclusion of two sets; for measures this means one measure is a restriction of
      another.
 \item They admit glueing operation. For example, suppose
       $\{\calP_i\mid i\in I\}$ is a family of patch
       such that if $i,j\in I, T\in\calP_i$ and $S\in\calP_j$, then either
       $S=T$ or $S\cap T=\emptyset$.
       Then we can ``glue'' $\calP_i$'s and obtain a patch $\bigcup_{i\in I}\calP_i$.
 \item There are ``zero elements'', which contains nothing.
       For example, empty set is a patch that contains no tiles; zero function also
       contains no information.
         Such a zero element is often unique for each category
       of objects. 
\end{enumerate}

In Section \ref{section_pattern_space} we study these operations in an abstract
setting. In Subsection \ref{subsection_def_pat_sp}
we first find a set of axioms that the cutting-off operations should satisfy
(Definition \ref{def_pattern_sp}). Several concrete cutting-off operations
of objects such as patches and  point sets are proved to satisfy this axiom.
The sets with such cutting-off operations are called \textit{abstract pattern spaces}
and elements such as tilings and Delone sets in abstract pattern spaces are called
\textit{abstract patterns}. In the rest of this section we study the rest of the
structures given above by capturing them by cutting-off operation.
In Subsection \ref{subsection_order_pat_sp}, we study an order relation between
two abstract patterns. This is an abstract notion which captures ``inclusion'' in the
above list. This relation also gives a way to capture  the operation of
``taking the union'' of the list,
and in Subsection \ref{subsection_glueable_pat_sp} we define a
\textit{glueable abstract pattern space}, where we can ``often'' glue objects.
There we also show several examples of abstract pattern spaces are glueable.
In Subsection \ref{subsection_zero_elements} we define zero elements in an abstract
setting and give a sufficient condition for its uniqueness.

Abstract pattern spaces often admit group actions by the group of isometries of the
ambient space where abstract patterns live.
We study abstract pattern spaces with group actions in Section
\ref{section_gammma-pattern_space}.
In Subsection \ref{subsection_def_Gamma-pat_sp} we give the axiom that such group
actions should satisfy and examples.

In Subsection \ref{subsection_local_derivable} we define local derivability between
two abstract patterns, using the cutting-off operation and the group action.
We will prove that this is a generalization of local derivability and MLD
in the literature under a
mild assumption, by using the structures on abstract pattern spaces given above.

As was mentioned earlier, we prove MLD between various abstract patterns.
We answer the following two questions:
\begin{problem}\label{problem}
\begin{enumerate}
 \item There are several canonical maps, such as
        \begin{enumerate}
	 \item the map that sends a Delone set $D$ in a proper metric space $X$ to
                 a positive measure $\sum_{x\in D}\delta_x$, where
                $\delta_x$ is the Dirac measure at a point  $x$,
        \item  the map that sends a continuous bounded function $f$ on a 
              locally compact abelian group $G$ to
                a measure $fd\mu$,
                 where $\mu$ is a Haar measure,
	\end{enumerate}
        and so on. Do these maps send an object $\calP$ to  one that is MLD with
        $\calP$?
 \item For an abstract pattern
       $\calP$ and an interesting class $\Sigma$ of abstract patterns,
        can we describe a condition on $\calP$ and $\Sigma$ that assures that there is 
          $\calQ\in\Sigma$ which is MLD with $\calP$?
\end{enumerate} 
\end{problem}

We solve the first question affirmatively, in Subsection \ref{subsection_local_derivable}:
 see 
Proposition \ref{prop_MLD_D_sum_of_dirac_delta},
Proposition \ref{prop_MLD_f_fdmu},
and Proposition \ref{prop_MLD_voronoi}.
These show our generalized MLD is a natural concept.

We address the second question in Section \ref{section_trans_thm}. See
Theorem \ref{translation_thm}.
We use the structures listed above to
prove that a condition on $\calP$ and  one on $\Sigma$
(not on the relations between $\calP$ and $\Sigma$)
 assures that there exists a $\calQ\in\Sigma$ which is MLD with $\calP$.
 The conditions are mild enough so that many interesting examples
satisfy them.
In particular, many abstract patterns are MLD (with or without rotation)
with (non-multi) Delone sets (Corollary \ref{cor_MLD_with_Delone}).
Although this MLD can be proved in an ad-hoc fashion for
many abstract patterns, there seemed  to be no general treatment.
Corollary \ref{cor_MLD_with_Delone} gives a sufficient condition for an abstract pattern
to have a Delone set that is MLD with the abstract pattern.
The first three subsections in Section \ref{section_trans_thm} are preliminary
results for Theorem \ref{translation_thm}, which is proved in Subsection
\ref{subsection_proof_translation_thm}.

Our definition of local derivability defines a graph, consisting of vertices of abstract
patterns and edges of local derivability, and it is fruitful to
study the structure of this graph.
As an example of this direction of research,
in Section \ref{section_application_translation_thm}
 we study pattern equivariant functions, via a study of this graph
 (Lemma \ref{lem_max_element_sigma_coincide}). Pattern-equivariant functions
 were first defined by
Kellendonk \cite{MR1985494}. Later Rand \cite{de2006pattern} generalized Kellendonk's
definition to incorporate rotation or $\Od$-actions.
We first show that, in each of definitions,
 pattern equivariant functions for an abstract pattern $\calP$
are the functions that are locally derivable from $\calP$.
Next, we show that two abstract patterns $\calP$ and $\calQ$
are MLD if and only if the spaces of
the pattern-equivariant functions are the same, under a mild condition, in a quite
general setting.
As the referee pointed out, for FLC Delone sets and FLC tilings in $\Rd$, if we do not
consider $\Od$ action, this type of result is
already known, but we extend it to two directions: we generalize the Kellendonk's
definition to a more general space $X$ and a group action $\Gamma\curvearrowright X$ and
prove this type of result (Theorem \ref{thm1_pat-equi_remembers}); we then prove this type
of result for Rand's definition (Theorem \ref{thm_pat_equiv_ft_remembers_P}).
The space of pattern equivariant functions has all the information of the original
abstract pattern up to MLD; in order to analyze certain abstract patterns up to MLD, it
suffices to investigate its space of pattern equivariant functions.
Such graph-theoretic argument of studying abstract patterns via arrows of local
derivability may have other applications.


Finally let us mention whether our argument is topological or metrical.
Our argument is metrical and depends on the choice of metric.
For example, in the literature
a subset $D$ of $\Rd$ is relatively dense if there is a compact subset $K$
of $\Rd$ such that $D+K=\Rd$.
This definition makes sense if we replace $\Rd$ with a locally compact abelian group.
However, for general topological spaces this does not make sense since there is no
group structure available. We have to assume the space admits either a metric or a
group action in order to define relative denseness. In this article we assume the
existence of metric for the spaces $X$ where abstract patterns
 such as tilings live and say a subset
$D$ of $X$ is relatively dense if there is $R>0$ such that any balls in $X$ of radius
$R$ contain points in $D$.
By using metric structure, we can define other useful notions, such as ``uniformly
discrete'', which means the distances between two points in a set $D\subset X$ is bounded
from below, and for abstract patterns to ``consists of bounded components''
(Definition \ref{def_bounded_coponents}), which means the diameter of tiles are bounded
from above if the abstract patterns are tilings.
We use these metric-dependent notions throughout the article.
In particular, we limit the relevance of this article to the case where
abstract patterns consist of bounded components.

However, it is desirable to put a topological assumption on the metrics.
As the referee pointed out, if we consider a metric $\rho'(x,y)=\min\{1,\rho(x,y)\}$
of $\Rd$, where $\rho$ is the standard Euclidean
metric, any non-empty subset $D$ of $\Rd$ is
relatively dense with respect to our definition, which contradicts the standard
definition of relative denseness. Thus \emph{we always assume that the metrics we consider are
proper}, a topological condition on the metrics. By this assumption, some definitions
that use a metric become equivalent to a topological notion: our definition of
relative denseness coincides with the topological definition given above and
our definition of local derivability (Definition \ref{def_local_derive}) is equivalent
to a topological definition under a mild assumption (Lemma \ref{lem_local_derivability}).

\begin{nota}\label{notation}
 For a metric space $(X,\rho)$, the \emph{closed} ball with its center $x\in X$ and its radius $r>0$
     is denoted by $B(x,r)$; that is, $B(x,r)=\{y\in X\mid \rho(x,y)\leqq r\}$.
 As was mentioned above, we assume every metric we consider on topological spaces is
 assumed to be proper, which means all closed balls are compact.
\footnote{Note that every second countable locally compact group admits a left-invariant
 proper metric (\cite{MR0348037}).}

      For a positive integer $d$, 
      let $\rho$ be the Euclidean metric for the Euclidean space $\Rd$.
     Let $\Ed$ be the group of all isometries on the Euclidean space
     $\Rd$ 
      and $\Od$ be the orthogonal group.
      There is a group isomorphism $\Rd\rtimes\Od\rightarrow\Ed$, by which we can identify
     these two groups. Thus elements of $\Ed$ are recognized as pairs $(a,A)$ of $a\in\Rd$
     and
      $A\in\Od$. For $\Ed$, define
      a metric $\rho_{\Ed}$ by 
     $\rho_{\Ed}((a,A),(b,B))=\rho(a,b)+\|A-B\|$, where
      $\|\cdot\|$ is the operator norm
     for the operators on the Banach space $\Rd$ with the Euclidean norm.
     For any closed subgroup $\Gamma$ of $\Ed$,
    the restriction $\rho_{\Gamma}$ of
     $\rho_{\Ed}$is a left-invariant metric for $\Gamma$.
     Moreover, for any $\gamma,\eta\in\Gamma$, we have
     \begin{align}
            \rho(\gamma 0,\eta 0)\leqq\rho_{\Gamma}(\gamma,\eta)\leqq
              \rho(\gamma 0,\eta 0)+2.\label{eq_relation_rho_rhoGamma}
     \end{align}


     We set $\mathbb{T}=\{z\in\mathbb{C}\mid |z|=1\}$.

     For any group $\Gamma$ which acts on a set $X$, its isotropy group for a point $x\in X$
     is denoted by $\Gamma_x$. That is, $\Gamma_x=\{\gamma\in\Gamma\mid\gamma x=x\}$.
     The identity element of any group is denoted by $e$.
     If $\calP$ is an object such as a patch, a function, a measure or a
     subset of $X$, its
     group of symmetry is by definition
     $\Sym_{\Gamma}\calP=\{\gamma\in\Gamma\mid\gamma\calP=\calP\}$
     (a special case of isotropy groups).
\end{nota}

\section{General theory of abstract pattern spaces}
\label{section_pattern_space}
\emph{In this section $X$ represents a nonempty topological space unless otherwise stated.}
First, in Subsection \ref{subsection_def_pat_sp}, we define ``abstract pattern space''.
Several spaces such as the space of patches and the space of subsets of $\Rd$
have an operation
of ``cutting off'': for example, for a discrete
 set $D\subset\Rd$ and a subset $C$ of $\Rd$,
we can ``cut off'' $D$ by the window $C$ by taking intersection $D\cap C$.
We axiomatize the properties that such cutting-off operation should have and
obtain the notion of abstract pattern space.
Several spaces of objects such as patches, subsets of $\Rd$, functions and measures are
captured in this framework.
In Subsection \ref{subsection_order_pat_sp} we introduce an order relation on a
abstract pattern space, which is the inclusion between two patches when the abstract pattern space
is the set of all patches or the set of all subsets of the ambient space.
In Subsection \ref{subsection_glueable_pat_sp} we study the operation of
``gluing'' objects to obtain a new object. 
This is an abstract framework to capture the usual operation of taking 
union.
Finally, in Subsection \ref{subsection_zero_elements}
we define zero elements, which is the empty-set in the abstract pattern space of all patches and
is zero function in the abstract pattern space of all functions.

\subsection{Definition and examples of abstract pattern space}
\label{subsection_def_pat_sp}

Here we define the framework of ``abstract pattern space'' for objects such as tilings and
Delone sets.
\begin{nota}
       The set of all closed subsets of $X$ is denoted by $\calC(X)$.
\end{nota}

\begin{defi}\label{def_pattern_sp}
      A non-empty set $\Pi$ equipped with a map
      \begin{align}
           \Pi\times\calC (X)\ni (\mathcal{P},C)\mapsto \calP\sci C\in\Pi \label{scissors_operation}
      \end{align}
      such that
      \begin{enumerate}
       \item $(\calP\sci C_1)\sci C_2=\calP\sci (C_1\cap C_2)$ for any $\calP\in\Pi$ and 
                any $C_1,C_2\in\calC(X)$, and
       \item for any $\calP\in\Pi$ there exists $C_{\calP}\in\calC(X)$ such that
             \begin{align*}
	         \calP\sci C=\calP\iff C\supset C_{\calP},
	     \end{align*}
               for any $C\in\calC(X)$,
      \end{enumerate}
      is called a \emph{abstract pattern space over $X$}.
      The map (\ref{scissors_operation}) is called the \emph{cutting-off operation} of
      the abstract pattern space $\Pi$.
      The closed set $C_{\calP}$ that appears in 2.\  is unique.
      It is called the \emph{support} of $\calP$ and is represented by $\supp\calP$.
      Elements in $\Pi$ are called \emph{abstract patterns} in $\Pi$.
\end{defi}


\begin{rem}
    Note that the symbol $\cap$ in the first axiom of abstract pattern space is the intersection
 of two sets. Note also that if $A\supset B$, $A$ and $B$ may be equal.
\end{rem}

\begin{rem}
    It is sometimes impossible to recover $\supp\calP$ from the information of
    $\supp(\calP\sci K)$, where $K$ runs through the set of all compact subsets of $X$:
    consider a non-compact $X$ and an abstract pattern space $\Pattern(X)$ (Definition
     \ref{def_patterns});
   The abstract pattern $\{X\}$ in this pattern
    space satisfies the condition $\{X\}\sci K=\emptyset$ for all compact $K$.
 Often we can recover $\supp\calP$; in fact
   if $\calP$ consists of  bounded components (Definition \ref{def_bounded_coponents}),
 we can do so.   
\end{rem}

\begin{lem}\label{lemma_support_calP_sci_C}
    Let $\Pi$ be an abstract pattern space over $X$.
    For any $\calP\in\Pi$ and $C\in\calC (X)$, we have $\supp(\calP\sci C)\subset(\supp\calP)\cap C$.
\end{lem}
\begin{proof}
\begin{align*}
   (\calP\sci C)\sci ((\supp\calP)\cap C)=(\calP\sci\supp\calP)\sci C=\calP\sci C.
\end{align*}
\end{proof}

\begin{rem}
    The inclusion in Lemma \ref{lemma_support_calP_sci_C}
    may be strict. In fact, in the abstract pattern space $\Patch(X)$ of all patches
   (Example \ref{example_patch}), if compared to the tiles in a patch $\calP$, a
   closed set $C$ is too small, then $\calP\sci C=\emptyset$ and so
    $\supp(\calP\sci C)=\emptyset$. On the other hand, $\supp\calP$ and $C$ may have
    non-empty intersection even if $C$ is small.
     For example, consider a tiling $\calP=\{(0,1)^d+x\mid x\in\mathbb{Z}^d\}$ and
      $C=B(0,1/2)$; then $\calP\sci C=\emptyset$ and $(\supp\calP)\cap C=C$.
\end{rem}

We now list several examples of abstract pattern space.
\begin{ex}[The space of patches in a metric space]\label{example_patch}
     Let $X$ be a (proper) metric space. An open, nonempty and bounded subset of $X$ is called a tile (in $X$).
     A set $\calP$ of tiles such that if $S,T\in\calP$, then either $S=T$ or $S\cap T=\emptyset$
     is called a patch (in $X$).
     The set of all patches in $X$ is denoted by $\Patch(X)$.
     For $\calP\in\Patch(X)$ and $C\in\calC(X)$, set
     \begin{align}
           \calP\sci C=\{T\in\calP\mid T\subset C\}.\label{def_scissor_for_patch}
     \end{align}
     With this cutting-off operation $\Patch(X)$ becomes an abstract pattern space over $X$.
     For $\calP\in\Patch(X)$, its support is
     \begin{align*}
           \supp\calP=\overline{\bigcup_{T\in\calP}T}.
     \end{align*}  
       Patches $\calP$ with $\supp\calP=X$ are called tilings.
\end{ex}

\begin{rem}
     Usually tiles are defined to be (1) a compact set that is the closure of its
 interior \cite{BH},
     or in Euclidean case,    (2) a polygonal subset of $\Rd$ \cite{Wh} or
 (3) a homeomorphic image of closed unit
 ball (for example, \cite{AP}).
     The advantage of our definition is that we can give punctures to tiles
    and we do not need to consider labels (Example \ref{ex_L-labeled_tiling}), and
     so we may avoid a slight abuse of language such as ``tiles $T$ and $S$ have disjoint
     interiors'' and simplify the notation.
     For example, we can define Robinson triangles (\cite[p.537]{MR857454})
      as the following four tiles:
     (1) the interior of triangle with side-length $\tau, \tau, 1$
      (where $\tau=\frac{1+\sqrt{5}}{2}$), with one point
       on the left-hand side removed, (2) the
 similar open set but one point on the right-hand
      side removed, (3) the interior of triangle with side-length $1,1,\tau$ with
      one point on the right-hand side removed, and (4) the similar open set but one point
       on the left-hand side removed.
        Giving punctures is also useful when we construct Voronoi tilings
       in Subsection \ref{subsection_local_derivable}, since in this case giving
      puncture is simpler than giving labels.

      The usual labeled tilings (Example \ref{ex_L-labeled_tiling})
      are often MLD with tilings with open tiles (Example \ref{example_patch}),
    and so in this article we mainly deal with
       tilings with  open tiles.
\end{rem}

\begin{ex}[The space of labeled patches, \cite{MR1976605}, \cite{MR2851885}]\label{ex_L-labeled_tiling}
      Let $L$ be a set. An $L$-labeled tile is a pair $(T,l)$ of a compact subset $T$
       of $X$ and $l\in L$, such that $T=\overline{T^{\circ}}$ (the closure of the
       interior).
       An $L$-labeled patch is a collection $\calP$ of $L$-labeled tiles such that
        if $(T,l), (S,k)\in\calP$, then either $T^{\circ}\cap S^{\circ}=\emptyset$, or
	$S=T$ and $l=k$. For an $L$-labeled patch $\calP$, define the support of $\calP$
       via
       \begin{align*}
	    \supp \calP=\overline{\bigcup_{(T,l)\in\calP}T}.
       \end{align*}
       An $L$-labeled patch $\calT$ with $\supp\calT=X$ is called an $L$-labeled tiling.
      Sometimes we suppress $L$ and call such tilings labeled tilings.

      For an $L$-labeled patch $\calP$ and $C\in\calC(X)$, define a cutting-off
       operation via
       \begin{align*}
	      \calP\sci C=\{(T,l)\in\calP\mid T\subset C\}.
       \end{align*}
       The space $\Patch_L(X)$ of all $L$-labeled patches is a pattern space over $X$
        with this cutting-off operation. 
\end{ex}

\begin{rem}
      There is another  operation of ``cutting off'' $L$-labeled patches, which is
     defined via
      \begin{align*}
             \calP\sqcap C=\{(T,l)\in\calP\mid T\cap C\neq\emptyset\}.
      \end{align*}
       However this does not define the structure of pattern space, as there is no
      unique support.
      Since we often assume the first condition of
    Assumption \ref{simplicity_assumption}, in most of the
 cases these two
      operations $\sci$ and $\sqcap$ are essentially the same.
\end{rem}

\begin{ex}[The space of patterns]\label{def_patterns}
      A set of non-empty subsets of $X$ is called a pattern
     (\cite[p.127]{baake2013aperiodic})\footnote{In \cite{baake2013aperiodic} patterns
      are assumed to be non-empty, but it is useful to include the  empty set as a pattern.}.
      The set of all patterns in $X$ is denoted by $\Pattern(X)$.
      $\Pattern(X)$ is an abstract pattern space over $X$ by the cutting-off operation
     defined via \eqref{def_scissor_for_patch}.
\end{ex}

There is another operation of ``cutting off'' pattern $\calP$, as follows:
\begin{defi}\label{def_sqcap}
     For a pattern $\calP$ in $\Rd$ and $C\subset\Rd$, we set
        \begin{align*}
	      \calP\sqcap C=\{T\in\calP\mid T\cap C\neq\emptyset\}.
       \end{align*}
\end{defi}
However this operation does not define a pattern space, since there is no unique
support.

\begin{ex}[The space of all locally finite subsets of a metric space]\label{example_LF(X)}
     Let $X$ be a metric space. 
     Let $\LF(X)$ be the set of all locally finite subsets of $X$; that is,
     \begin{align*}
           \LF(X)=\{D\subset X\mid \text{for all $x\in X$ and $r>0$, $D\cap B(x,r)$ is finite}\}.
     \end{align*}
     With the
     usual intersection
     $\LF(X)\times\calC(X)\ni(D,C)\mapsto D\cap C\in\LF(X)$ of two subsets
     of $X$ as a cutting-off operation, $\LF(X)$ is an abstract pattern space over $X$.
     For any $D\in\LF(X)$, its support is $D$ itself.     
\end{ex}

\begin{ex}[The space of all uniformly discrete subsets]\label{example_UD}
      We say, for $r>0$,
     a subset $D$ of a metric space $(X,\rho)$ is $r$-uniformly discrete if
     $\rho(x,y)>r$ for any $x,y\in D$ with $x\neq y$.
     The set $\UD_r(X)$ of all $r$-uniformly discrete subsets of $X$ is an abstract pattern space over $X$
     by the usual intersection as a cutting-off operation.
     If $D$ is $r$-uniformly discrete for some $r>0$, we say $D$ is uniformly discrete.
     The set $\UD(X)=\bigcup_{r>0}\UD_r(X)$ of all uniformly discrete subsets of $X$ is
     also an abstract pattern space over $X$.

     Subsets $D$ of $X$ that are uniformly discrete and relatively dense in $X$ are
 called \emph{Delone sets}. ``Relatively dense'' is defined as follows.
     For $R>0$, a subset $D\subset X$ is $R$-relatively
 dense if
        $D\cap B(x.R)^{\circ}\neq\emptyset$ for each $x\in X$.
        A subset $D$ of $X$ is relatively dense if it is $R$-relatively dense for
       some $R$. For $X=\Rd$ with the standard Euclidean metric, this definition is
      equivalent to the usual one (\cite[p.12]{baake2013aperiodic}).
\end{ex}

\begin{ex}\label{example_2^X_calC(X)}
     With the usual intersection of two subsets of $X$ as a cutting-off operation,
     the set $2^X$ of all subsets of $X$ and $\calC(X)$ are abstract pattern spaces over $X$.
      For $A\in 2^X$, the support $\supp A$ is the closure of $A$.
 
      For example, the union of all Ammann bars (\cite{MR857454}) in a
 Penrose tiling is an abstract pattern.
\end{ex}

The following example plays an important role when we discuss pattern-equivariant
functions in Section \ref{section_application_translation_thm}, because
pattern-equivariant functions are functions that are locally derivable from the
original abstract pattern. We later assume $Y$ has a topology and consider the
space of all continuous maps from $X$ to $Y$, but we need to consider the space
of all maps as follows, since after cutting-off, continuous maps may become
discontinuous.
\begin{ex}[The space of maps]\label{example_map}
      Let $Y$ be a nonempty set.
      Take one element $y_0\in Y$ and fix it.
      The abstract pattern space $\Map(X,Y,y_0)$ is defined as follows:
      as a set the space is equal to $\Map(X,Y)$ of all mappings from $X$ to $Y$;
      for $f\in\Map(X,Y,y_0)$ and $C\in\calC(X)$, the cutting-off operation is defined by
      \begin{align*}
           (f\sci C)(x)=
           \begin{cases}
	         f(x)&\text{if $x\in C$}\\
                 y_0&\text{if $x\notin C$}.
	   \end{cases}
      \end{align*}
      With this operation $\Map(X,Y,y_0)$ is an abstract pattern space over $X$
      and for $f\in\Map(X,Y,y_0)$ its support is 
      $\supp f=\overline{\{x\in X\mid f(x)\neq y_0\}}$.
\end{ex}

\begin{ex}[The space of measures]\label{ex_space_of_measures}
     Let $X$ be a locally compact $\sigma$-compact metric space.
     Let $C_c(X)$ be the space of all continuous and complex-valued functions on $X$
     which have compact supports.
     Its dual space $C_c(X)^*$ with respect to the inductive limit
     topology consists of Radon charges, that is,
     the maps $\Phi
     \colon C_c(X)\rightarrow\mathbb{C}$ such that there is a unique positive
     Borel measure $m$ and a Borel measurable map $u\colon X\rightarrow\mathbb{T}$
     such that
     \begin{align*}
      \Phi(\varphi)=\int_{X}\varphi udm
     \end{align*}
     for all $\varphi\in C_c(X)$.
     For such $\Phi$ and $C\in\calC (X)$ set
     \begin{align*}
         (\Phi\sci C)(\varphi)=\int_C \varphi udm
     \end{align*}     
     for each $\varphi\in C_c(X)$.
     Then the new functional $\Phi\sci C$ is a Radon charge.
     With this operation $C_c(X)^*\times\calC (X)\ni(\Phi,C)\mapsto
      \Phi\sci C\in C_c(X)^*$,
     the space $C_c(X)^*$ becomes an abstract pattern space over $X$.

      Note that if $m$ is a positive measure on $X$ and $u\colon X\rightarrow\mathbb{C}$
       is a bounded Borel map (not necessarily $\mathbb{T}$-valued), 
       then $\Phi\colon C_c(X)\ni\varphi\mapsto\int\varphi udm$ 
       is a Radon charge. If $C\in\calC(X)$, then
        \begin{align*}
	     (\Phi\sci C)(\varphi)=\int_C\varphi udm,
	\end{align*}
         for each $\varphi\in C_c(X)$.

      Note also that if $X$ is second countable and $\mu$ is a positive measure,
      the topological support of $\mu$ coincides with the support of the functional
       $C_c(X)\ni\varphi\mapsto\int\varphi d\mu$ as an abstract pattern.
\end{ex}

Next we investigate abstract pattern subspaces.
The relation between an abstract pattern space and its abstract
pattern subspaces is similar to the one
between  a set with a group action and its invariant subsets.

\begin{defi}
     Let $\Pi$ be an abstract pattern space over $X$.
     Suppose a non-empty subset $\Pi'$ of $\Pi$ satisfies the condition
     \begin{align*}
          \text{$\calP\in\Pi'$ and $C\in\calC(X)\Rightarrow\calP\sci C\in\Pi'$}.
     \end{align*}
     Then $\Pi'$ is called an \emph{abstract pattern subspace} of $\calP$.
\end{defi}

\begin{rem}
      If $\Pi'$ is an abstract
      pattern subspace of an abstract pattern space $\Pi$, then $\Pi'$ is a 
      abstract pattern space by restricting the cutting-off operation.
\end{rem}

\begin{ex}
      Let $X$ be a topological space. Then $\calC(X)$ is an abstract pattern subspace of
      $2^X$.
      If $X$ is a metric space, then $\LF(X)$ is an abstract
 pattern subspace of $\calC(X)$ and
           $\UD_r(X)$ is an abstract pattern subspace of $\UD(X)$ 
      for each $r>0$. 
       Since we assume the metrics we consider are proper,
      $\UD(X)$ is an abstract pattern subspace of
      $\LF(X)$.
\end{ex}

Next we investigate two ways to construct new abstract pattern spaces from
old ones; taking product and taking power set.

\begin{lem}\label{lemmma_product_pattern_space}
      Let $\Lambda$ be an index set and $\Pi_{\lambda},\lambda\in\Lambda$, is a 
     family of abstract pattern spaces over $X$.
     The direct product $\prod_{\lambda}\Pi_{\lambda}$ becomes an abstract pattern space over $X$
 with 
     the  cutting-off operation defined via
     \begin{align*}
      (\calP_{\lambda})_{\lambda\in\Lambda}\sci 
        C=(\calP_{\lambda}\sci C)_{\lambda\in\Lambda}.
     \end{align*}
     for $(\calP_{\lambda})_{\lambda}\in\prod_{\lambda}\Pi_{\lambda}$ and $C\in\calC(X)$.
     The support is given by
     $\supp(\calP_{\lambda})_{\lambda}=\overline{\bigcup_{\lambda}\supp\calP_{\lambda}}$.
\end{lem}

\begin{defi}\label{def_product_pattern_sp}
    Under the same condition as in Lemma \ref{lemmma_product_pattern_space}, we call
    $\prod\Pi_{\lambda}$ the \emph{product abstract pattern space} of $(\Pi_{\lambda})_{\lambda}$.
\end{defi}
This construction of product abstract pattern space will be used in Proposition
\ref{prop_calP_MLD_Gamma_lambda}. The following construction of Delone multi set,
which uses product, is
also essential.

\begin{ex}[uniformly discrete multi set, \cite{MR1976605}]\label{ex_multi_set}
     Let $I$ be a set. Consider the abstract pattern subspace $\UD^I(X)$ of
 $\prod_{i\in I}\UD(X)$, defined via
     \begin{align*}
          \UD^I(X)=\{(D_i)_{i\in I}\mid \bigcup_i D_i\in\UD(X)\}.
     \end{align*}
     Elements of $\UD^I(X)$ are called uniformly discrete multi sets.
      A uniformly discrete multi set $(D_i)_i\in\UD^I(X)$ is called a Delone multi set
     if each $D_i$ is a Delone set and the union $\bigcup_iD_i$ is a Delone set.
\end{ex}

The following yet another construction of abstract pattern space will be useful when
we deal with ``glueing'' of abstract patterns
(Subsection \ref{subsection_glueable_pat_sp}).
\begin{lem}
     Let $\Pi$ be an abstract pattern space over $X$.
     The set $2^{\Pi}$ of all subsets of $\Pi$ is an abstract pattern space over $X$ with the
 cutting-off operation defined via
     \begin{align}
         \Xi\sci C=\{\calP\sci C\mid \calP\in\Xi\},\label{eq_scissors_operation_powewset}
     \end{align}
     for any $\Xi\in 2^{\Pi}$ and $C\in\calC(X)$.
     The support is given by $\supp\Xi=\overline{\bigcup_{\calP\in\Xi}\supp\calP}$.
\end{lem}

\begin{defi}\label{Def_power_pattern_space}
     The power set $2^{\Pi}$ of an abstract pattern space $\Pi$, endowed with the cutting-off operation
      in equation (\ref{eq_scissors_operation_powewset}),
      is called the \emph{power abstract pattern space} 
      of $\Pi$.
\end{defi}

Note that the power pattern space
 $2^{2^X}$ is not $\Pattern (X)$, since $2^X$ includes the empty set.

Next, we define a notion which will be useful later.
Maps and elements of $2^X$ (and so uniformly discrete subsets of $X$) always 
satisfy this condition;
a patch (and so a tiling) satisfies this condition if and only if
the diameters of tiles in that patch are bounded from above.
This is one of the conditions in Assumption \ref{simplicity_assumption}, which we
frequently assume.

\begin{defi}\label{def_bounded_coponents}
      Let $\Pi$ be an abstract pattern space over a metric space $X$.
     For any element $\calP\in\Pi$,  we say $\calP$ \emph{consists of 
       bounded components} if there is $R_{\calP}>0$ such that
      for any $C\in\calC(X)$ and  $x\in\supp(\calP\sci C)$,
    we have $x\in\supp(\calP\sci C\sci B(x,R_{\calP}))$.
\end{defi}

\subsection{An order on abstract pattern spaces}
\label{subsection_order_pat_sp}

Here we introduce an order relation $\geqq$ on abstract pattern spaces, which captures
``inclusion'' between abstract patterns in a general context.

\begin{defi}\label{def_order_pattern_space}
     Let $\Pi$ be an abstract pattern space over $X$. We define a relation
      $\geqq$ on $\Pi$ as follows:
     for each $\calP,\calQ\in\Pi$, we set $\calP\geqq\calQ$ if 
     \begin{align*}
          \calP\sci\supp\calQ=\calQ.
     \end{align*}
\end{defi}

The following two lemmas will be used throughout the article.

\begin{lem}
     \begin{enumerate}
      \item If $\calP\geqq\calQ$, then $\supp\calP\supset\supp\calQ$.
      \item The relation $\geqq$
             is an order on $\Pi$.
     \end{enumerate}
\end{lem}
\begin{proof}
      If $\calP\geqq\calQ$, then
      \begin{align*}
            \calQ\sci\supp\calP=\calP\sci\supp\calP\sci\supp\calQ=\calP\sci\supp
            \calQ=\calQ.
      \end{align*}
      Thus $\supp\calP\supset\supp\calQ$.
      Next we prove that $\geqq$ is an order. $\calP\geqq\calP$ is clear.
      If $\calP\geqq\calQ$ and $\calQ\geqq\calP$, then
      $\supp\calP=\supp\calQ$ and $\calP=\calP\sci\supp\calP=\calP\sci\supp\calQ=\calQ$.
      Finally, if $\calP\geqq\calQ\geqq\calR$, then
      $\supp\calP\supset\supp\calQ\supset\supp\calR$ and
      $\calP\sci\supp\calR=\calP\sci\supp\calQ\sci\supp\calR=\calQ\sci\supp\calR=\calR$,
      and so $\calP\geqq\calR$.
\end{proof}

\begin{lem}\label{lem_order_pattern_space}
     \begin{enumerate}
      \item If $\calP\in\Pi$ and $C\in\calC(X)$, then $\calP\geqq\calP\sci C$.
      \item If $\calP,\calQ\in\Pi$, $C\in\calC(X)$ and $\calP\geqq\calQ$, then
            $\calP\sci C\geqq\calQ\sci C$.
     \end{enumerate}      
\end{lem}
\begin{proof}
      The statements follow from  Lemma \ref{lemma_support_calP_sci_C}.

      1.   $\calP\sci\supp(\calP\sci C)=\calP\sci \supp\calP\sci C\sci \supp(\calP\sci C)
           =\calP\sci C\sci\supp(\calP\sci C)=\calP\sci C.$

      2. $\calP\sci C\sci \supp(\calQ\sci C)=\calP\sci\supp\calQ\sci C\sci\supp(\calQ\sci C)
          =\calQ\sci C$.
\end{proof}

The supremum of a subset $\Xi\subset\Pi$ with respect to the order $\geqq$ is the
``union'' of abstract patterns $\calP\in\Xi$. It does not necessarily exist, but
is a important concept. Below we investigate elementary properties of supremum.

\begin{defi}
       Let $\Xi$ be a subset of $\Pi$.
       If the supremum of $\Xi$ with respect to the order $\geqq$ defined in 
       Definition \ref{def_order_pattern_space}
exists in $\Pi$,
       it is denoted by $\bigvee\Xi$.
\end{defi}

We briefly discuss relations between supremum and support.
\begin{lem}\label{lem_support_supremum}
       If a subset $\Xi\subset\Pi$ admits the supremum $\bigvee\Xi$, then
       $\supp\bigvee\Xi=\overline{\bigcup_{\calP\in\Xi}\supp\calP}$.
\end{lem}
\begin{proof}
       Set $C=\overline{\bigcup_{\calP\in\Xi}\supp\calP}$.
       Since $\bigvee\Xi\geqq\calP$ for any $\calP\in\Xi$, by 
       Lemma \ref{lem_order_pattern_space}
       $\supp\bigvee\Xi\supset\supp\calP$ for each $\calP\in\Xi$.
       Since the support is closed, we have $\supp\bigvee\Xi\supset
       C$.
 
        If we assume $\supp\bigvee\Xi$ is strictly larger than $C$, then  we have
        the following contradiction.
        Since $\supp((\bigvee\Xi)\sci C)\subset C\neq\supp\bigvee\Xi$,
        the two abstract patterns $\bigvee\Xi$ and $(\bigvee\Xi)\sci C$
        are different and $\bigvee\Xi\geqq(\bigvee\Xi)\sci C$
        by Lemma \ref{lem_order_pattern_space}.
        On the other hand, $(\bigvee\Xi)\sci C$ majorizes $\Xi$.
        These contradict the fact that $\bigvee\Xi$ is the supremum.
\end{proof}
\begin{ex}\label{example_no_supremum}
        It is not necessarily true that any element $\calP_{0}$ in $\Pi$
        that majorizes $\Xi$ and $\supp\calP_{0}=\overline{\bigcup_{\calP\in\Xi}\supp\calP}$
        is the supremum of $\Xi$. For example, let the abstract pattern space be
         $\Map([0,1],\mathbb{C},0)$ (Example \ref{example_map}). Set
         $\Xi=\{\delta_x\mid 0<x\leqq 1\}$, where $\delta_x$ is the Kronecker
       delta function.
       For each $a\in\mathbb{C}$ define the function $f_a$ via
        \begin{align*}
	      f_a(x)=\begin{cases}
		          a\text{ if } x=0\\
		          1\text{ if } x>0.
		     \end{cases}
	\end{align*}
        Then each $f_a$ is a upper bound for $\Xi$ with $\supp f_a=[0,1]$ but there is
 no order relation
      between $f_a$'s.

         Also note that the first part of the proof of Lemma \ref{lem_support_supremum}
       shows that, if $\calQ$ is an upper bound for $\Xi$, then
       $\supp\calQ\supset\overline{\bigcup_{\calP\in\Xi}\supp\calP}$.
\end{ex}

The following lemma will be useful later.

\begin{lem}\label{useful_lem_for_last_half_of_main_thm}
       Let $F_j$ be a finite subset of $X$ for $j=1,2$.
       Take a positive real number $r$ such that for each $j=1,2$,
       any two distinct elements $x,y\in F_j$ satisfy $\rho(x,y)>4r$.
       Suppose for each $j$ and $x\in F_j$, there corresponds $\calP_{x}^j\in\Pi$
      such that $\emptyset\neq\supp\calP_{x}^j\subset B(x,r)$.
      Suppose also there is $Q^j=\bigvee\{\calP_{x}^j\mid x\in F_j\}$ for $j=1,2$.
      Then the following statements hold:
      \begin{enumerate}
       \item If $\supp\calQ^1\subset\supp\calQ^2$, then for each $x\in F_1$ there is 
              a unique $y\in F_2$ such that
               $\supp\calP_{x}^1\cap\supp\calP_{y}^2\neq\emptyset$.
              In this case $\supp\calP_{x}^1\subset\supp\calP_{y}^2$ holds.
	\item If $\supp\calQ^1=\supp\calQ^2$, then for each $x\in F_1$ there is a 
	      unique $y\in F_2$ such that $\supp\calP_{x}^1=\supp\calP_{y}^2$.
	\item If $\calQ^1=\calQ^2$, then for each $x\in F_1$ there is a unique
	      $y\in F_2$ such that $\calP_{x}^1=\calP_{y}^2$.
      \end{enumerate}
\end{lem}
\begin{proof}
      1.
      By Lemma \ref{lem_support_supremum}, 
      $\supp\calQ^j=\bigcup_{x\in F_j}\supp\calP^{j}_{x}$ for each $j=1,2$.
      For each $x\in F_1$, there is $y\in F_2$ such that 
      $\supp\calP_{x}^1\cap\supp\calP_y^2\neq\emptyset$.
      If there is another $y'\in F_2$ such that
      $\supp\calP_x^1\cap\supp\calP_{y'}^2\neq\emptyset$, then
      $B(x,r)\cap B(y,r)\neq\emptyset$ and $B(x,r)\cap B(y',r)\neq\emptyset$ and so
      $\rho(y,y')\leqq 4r$. By definition of $r$, we have $y=y'$.
      This shows the uniqueness of $y$.
      Uniqueness implies the last statement.

      2. By 1., for each $x\in F_1$ there is $y\in F_2$ such that 
         $\supp\calP_x^1\subset\supp\calP_y^2$.
         Applying 1. again, there is $x'\in F_1$ such that 
         $\supp\calP_y^2\subset\supp\calP_{x'}^1$.
         We have $\supp\calP_x^1\subset\supp\calP_{x'}^1$ and by applying
         uniqueness in 1., we see $x=x'$.
     
      3. By 2., for each $x\in F_1$ there is $y\in F_2$ such that
         $\supp\calP_x^1=\supp\calP_y^2$.
         Then 
         \begin{align*}
	       \calP_x^1=\calQ^1\sci\supp\calP_1^x=
                        \calQ^2\sci\supp\calP_2^y
	                =\calP_y^2.
	 \end{align*}
          Uniqueness follows from uniqueness in 1.
\end{proof}

\subsection{Glueable abstract pattern spaces}
\label{subsection_glueable_pat_sp}
        \emph{In this subsection  $X$ is a (proper) metric space with a metric $\rho$
        and $\Pi$ is an abstract pattern space over $X$.}

        Often we want to ``glue'' or ``take the union of ''
 	abstract patterns to obtain a larger abstract pattern.
        For example, suppose
         $\Xi$ is a collection of patches such that if $\calP,\calQ\in\Xi$,
         $S\in\calP$ and $T\in\calQ$, then we have either $S=T$ or $S\cap T=\emptyset$.
         Then we can ``glue'' patches in $\Xi$, that is, we can take 
         the union $\bigcup_{\calP\in\Xi}\calP$, which is also a patch.
         Abstract pattern spaces in which we can ``glue'' abstract patterns are called
         glueable abstract pattern spaces (Definition \ref{def_glueable_pattern_space}).
	 We define glueable abstract pattern spaces after introducing necessary notions
	 and proving a lemma.
	   Examples are given in page \pageref{ex_patch_compatible}.

\begin{defi}\label{def_local_finite_compatible}
      \begin{enumerate}
       \item Two abstract patterns $\calP,\calQ\in\Pi$ are said to be \emph{compatible}
              if there is $\calR\in\Pi$ such that $\calR\geqq\calP$ and
              $\calR\geqq\calQ$.
       \item A subset $\Xi\subset\Pi$ is said to be \emph{pairwise compatible} if
             any two elements $\calP,\calQ\in\Pi$ are compatible.
        \item A subset $\Xi\subset\Pi$ is said to be \emph{locally finite}
               if for any $x\in X$ and $r>0$, the set
               $\Xi\sci B(x,r)$, which was defined in 
             (\ref{eq_scissors_operation_powewset}),
             is finite.
      \end{enumerate}
\end{defi}

\begin{rem}
       We will prove in many abstract pattern spaces, the locally finite and pairwise
 compatible subsets admit supremums.
        Due to Example \ref{example_no_supremum}, in the space $\Map(X,Y,y_0)$
       (Example \ref{example_map}), a
      pairwise compatible $\Xi\subset\Map(X,Y,y_0)$ need not have supremum.
      We have to assume
      local finiteness for a subset $\Xi$ to admit supremum.
     Also, in order for a subset of an abstract pattern space
    to have the supremum,
 we have to assume pairwise compatibility, since
        this follows from the existence of supremum.
\end{rem}

\begin{lem}\label{lem_locally_finite_compatible}
        Let $\Xi$ be a subset of $\Pi$ and take $C\in\calC(X)$.
        Then the following hold.
        \begin{enumerate}
	 \item If $\Xi$ is locally finite, then so is $\Xi\sci C$.
         \item If $\Xi$ is pairwise compatible, then so is $\Xi\sci C$.
	\end{enumerate}            
\end{lem}
\begin{proof}
         1. Suppose there are $x\in X$, $r>0$ such that $\Xi\sci C\sci B(x,r)$ is
         infinite. There are $\calP_1,\calP_2,\ldots$ in $\Xi$ such that
         $\calP_n\sci C\sci B(x,r)$ are all distinct. 
         However by local finiteness of $\Xi$, there are distinct $n$ and $m$ such that
         $\calP_n\sci B(x,r)=\calP_m\sci B(x,r)$;
         this implies that
         $\calP_n\sci C\sci B(x,r)=\calP_m\sci C\sci B(x,r)$ and leads to a
          contradiction.

        2. Take $\calP,\calQ\in\Xi$ arbitrarily.
           By Definition \ref{def_local_finite_compatible}, there is $\calR\in\Xi$
           such that $\calR\geqq\calP$ and $\calR\geqq\calQ$.
           By Lemma \ref{lem_order_pattern_space},
           we have $\calR\sci C\geqq\calP\sci C$ and $\calR\sci C \geqq\calQ\sci C$,
           and so $\calP\sci C$ and $\calQ\sci C$ are compatible.
\end{proof}

\begin{defi}\label{def_glueable_pattern_space}
          An abstract pattern space $\Pi$ over a metric space $X$ is said to be \emph{glueable}
          if the following two conditions hold:
          \begin{enumerate}
	   \item If $\Xi\subset\Pi$ is both locally finite and pairwise compatible,
                 then there is the supremum $\bigvee{\Xi}$ for $\Xi$.
           \item If $\Xi\subset\Pi$ is both locally finite and pairwise compatible,
                 then for any $C\in\calC(X)$,
                 \begin{align}
		     \bigvee(\Xi\sci C)=(\bigvee\Xi)\sci C.\label{eq_def_glueable}
		 \end{align}
	  \end{enumerate}
\end{defi}

\begin{rem}
       By Lemma \ref{lem_locally_finite_compatible}, for $\Xi\subset\Pi$
      which is locally finite and pairwise compatible 
       and $C\in\calC(X)$,
       the left-hand side of the equation (\ref{eq_def_glueable}) makes sense.

       The first condition of this definition does not imply the second.
       Here is the sketch of the construction of a counterexample:
       for any abstract pattern space $\Pi$ and an element $\calP\in\Pi$, there is the
       smallest abstract pattern subspace $\Pi(\calP)$ that contains $\calP$.
       Let $\Pi$ be the abstract pattern space $\Pi=\Pattern(\mathbb{R})$
       (Example \ref{def_patterns}). Consider a pattern
        \begin{align*}
	    \calP=\{(0,1)+n\mid n\in\mathbb{Z}\}\cup\left\{\left(\frac{1}{2},\frac{3}{2}\right)+n
	 \mid n\in\mathbb{Z}\right\}.
	\end{align*}
         Any subset $\Xi$ of $\Pi(\calP)$ admits the supremum, but
         for $\Xi=\{\{(0,1)\},\{(1,2)\}\}$,
       we have $\bigvee\Xi=\{(0,1),(1/2,3/2),(1,2)\}$, and so
        $(\bigvee\Xi)\sci [0,3/2]=\left\{(0,1),\left(1/2,3/2\right)\right\}$ but
      $\bigvee(\Xi\sci [0,3/2])=\{(0,1)\}$.
\end{rem}

Before listing examples of glueable abstract pattern spaces, we show the result of
``two-step gluing'' is the same as the result of ``gluing once''.
\begin{lem}\label{lem_family_of_loc_fin_pair_comp_sets}
       Let $\Pi$ be glueable and $\Lambda$ a set.
       For each $\lambda\in\Lambda$, let $\Xi_{\lambda}\subset\Pi$ be a subset and
       suppose $\bigcup_{\lambda}\Xi_{\lambda}$ is locally finite and
       pairwise compatible.
       Then for each $\lambda$, the set $\Xi_{\lambda}$ is locally finite and
       pairwise-compatible. Moreover, if we set $\calQ_{\lambda}=\bigvee\Xi_{\lambda}$,
       the set $\{\calQ_{\lambda}\mid \lambda\in\Lambda\}$ is locally finite
       and pairwise-compatible and
       \begin{align*}
	     \bigvee\bigcup_{\lambda}\Xi_{\lambda}=\bigvee\{\calQ_{\lambda}\mid\lambda\in\Lambda\}.
       \end{align*} 
\end{lem}
\begin{proof}
      Set $\calP=\bigvee\bigcup_{\lambda}\Xi_{\lambda}$.
      For each $\lambda\in\Lambda$ and $\calQ\in\Xi_{\lambda}$, we have
      $\calP\geqq\calQ$ and so $\calP\geqq\calQ_{\lambda}$.
      This, in particular, shows that $\{\calQ_{\lambda}\mid\lambda\}$ is 
       pairwise compatible. Moreover, since for each $x\in X$ and $r>0$,
       \begin{align}
	    \{\calQ_{\lambda}\sci B(x,r)\mid\lambda\in\Lambda\}=
                    \{\bigvee(\Xi_{\lambda}\sci B(x,r))\mid \lambda\in\Lambda\},
            \label{eq_lem_suishin}
       \end{align}
       $\Xi_{\lambda}\sci B(x,r)\subset(\bigcup_{\lambda}\Xi_{\lambda})\sci B(x,r)$ and 
       $(\bigcup\Xi_{\lambda})\sci B(x,r)$ is finite by assumption,
       the set (\ref{eq_lem_suishin}) is finite: the set $\{\calQ_{\lambda}\mid\lambda\}$
       is locally finite.

       If $\calP'$ is a majorant for $\{\calQ_{\lambda}\mid\lambda\}$, then $\calP'\geqq\calQ$
       for each $\lambda\in\Lambda$ and $\calQ\in\Xi_{\lambda}$, and so $\calP'\geqq\calP$.
       As was mentioned above, $\calP$ is a majorant for $\{\calQ_{\lambda}\mid\lambda\}$, and
       so it is its supremum.
\end{proof}

We finish this subsection with examples.
\begin{ex}\label{ex_patch_compatible}
      Consider $\Pi=\Patch(X)$ (Example \ref{example_patch}).
      In this abstract pattern space, for two elements $\calP,\calQ\in\Patch(X)$,
       the following statements hold:
      \begin{enumerate}
       \item  $\calP\geqq\calQ\iff\calP\supset\calQ$.
       \item  $\calP$ and $\calQ$ are compatible if and only if for any $T\in\calP$ and
              $S\in\calQ$, either $S=T$ or $S\cap T=\emptyset$ holds.
      \end{enumerate}
       If $\Xi\subset\Patch(X)$ is pairwise compatible, then
       $\calP_{\Xi}=\bigcup_{\calP\in\Xi}\calP$ is a patch,
        which is the supremum of $\Xi$. If $C\in\calC(X)$, then
        \begin{align*}
	      (\bigvee\Xi)\sci C=(\bigcup_{\calP\in\Xi}\calP)\sci C
             =\bigcup(\calP\sci C)=\bigvee(\Xi\sci C).
	\end{align*}    
         $\Patch(X)$ is glueable.
\end{ex}

\begin{ex}\label{ex_compatible_2^X}
       For the abstract pattern space $2^X$ in Example \ref{example_2^X_calC(X)},
        two elements $A,B\in 2^X$
       are compatible if and only if 
       \begin{align}
        \text{ $\overline{A}\cap B\subset A$ and
       $A\cap\overline{B}\subset B$.}\label{condition_compatible_2^X}
       \end{align}
        Indeed, 
       if $A$ and $B$ are compatible, then there is a majorant $C$.
       By $C\supset A\cup B$,
       \begin{align*}
	     A\cup(\overline{A}\cap B)=(A\cup B)\cap\overline{A}
             =C\cap \overline{A}\cap (A\cup B)
	     =A\cap (A\cup B)=A,
       \end{align*}
       and so $\overline{A}\cap B\subset A$.
       A similar argument shows that $\overline{B}\cap A\subset B$.
       Conversely, if the condition (\ref{condition_compatible_2^X}) holds, then
       $(A\cup B)\cap \overline{A}=A\cup (B\cap\overline{A})=A$ and similarly 
       $(A\cup B)\cap \overline{B}=B$, and so $A\cup B$ is a majorant for $A$ and $B$.

       Suppose $\Xi\subset 2^X$ is locally finite and pairwise compatible.
       Note that $\bigcup_{A\in\Xi}\overline{A}=\overline{\bigcup_{A\in\Xi}A}$.
       Set $A_{\Xi}=\bigcup_{A\in\Xi}A$. For each $A\in\Xi$,
       $A_{\Xi}\cap \overline{A}=\bigcup_{B\in\Xi}(B\cap\overline{A})=A$;
       $A_{\Xi}$ is a majorant of $\Xi$. If $B$ is also a majorant for $\Xi$,
      then
       \begin{align*}
	B\cap\overline{A_{\Xi}}=B\cap(\bigcup_{A\in\Xi}\overline{A})
                              =\bigcup_{A\in\Xi}(B\cap\overline{A})
                             =\bigcup_{A\in\Xi}A=A_{\Xi},
       \end{align*}
        and so $B\geqq A_{\Xi}$.
       It turns out that $A_{\Xi}$ is the supremum for $\Xi$.
       Moreover, if $C\in\calC(X)$, then
      $A_{\Xi}\sci C=\bigcup_{A\in\Xi}(A\cap C)=\bigvee(\Xi\sci C)$.
      Thus  $2^X$ is a glueable space.  
\end{ex}



 \begin{rem}
       Let $\Pi_0$ be a glueable abstract pattern space and $\Pi_1\subset\Pi_0$ a
       pattern subspace. For any subset $\Xi\subset\Pi_1$,
       if  it is pairwise compatible
      in $\Pi_1$, then it is pairwise compatible in $\Pi_0$.  
      Moreover, whether a set is locally finite or not is independent of the
      ambient abstract pattern space in which the set is included.

  For a subset $\Xi\subset\Pi_1$ which is locally finite and
     pairwise compatible in $\Pi_1$, since $\Pi_0$ is glueable,
     there is the supremum $\bigvee\Xi$ in $\Pi_0$.
     If this supremum in $\Pi_0$ is always included in 
     $\Pi_1$, then $\Pi_1$ is glueable.

     By this remark, it is easy to see the abstract pattern spaces $\calC(X)$ 
     (Example \ref{example_2^X_calC(X)}), $\LF(X)$
     (Example \ref{example_LF(X)}), and $\UD_r(X)$ (Example
      \ref{example_UD},
      $r$ is an arbitrary positive number) are
      glueable.

      However, $\UD(X)$ (Example \ref{example_UD}) is not necessarily glueable.
      For example, set $X=\mathbb{R}$.
      Set $\calP_n=\{n,n+\frac{1}{n}\}$ for each integer $n\neq 0$.
      Each $\calP_n$ is in $\UD(\mathbb{R})$, $\Xi=\{\calP_n\mid n\neq 0\}$
      is locally finite and pairwise compatible, but it does not admit the supremum.
 \end{rem}

For the rest of this subsection we show that $\Map(X,Y,y_0)$ (Example \ref{example_map})
is glueable, where $X$ is a metric space, $Y$ a set and $y_0\in Y$.
This is proved in Proposition \ref{prop_map_glueable}, after proving preliminary
technical lemmas.

\begin{lem}\label{lem1_map_glueable}
      Two maps $f,g\in\Map(X,Y,y_0)$ are compatible if and only if
       $f|_{\supp f\cap\supp g}=g|_{\supp f\cap\supp g}$.
\end{lem}
\begin{proof}
       Suppose $f$ and $g$ are compatible. Take a majorant $h\in\Map(X,Y,y_0)$.
       For each $x\in\supp f\cap\supp g$,
       \begin{align*}
	     f(x)=(h\sci\supp f)(x)=h(x)=(h\sci\supp g)(x)=g(x).
       \end{align*}
       Conversely suppose $f|_{\supp f\cap\supp g}=g|_{\supp f\cap\supp g}$.
       Define a map $h\in\Map(X,Y,y_0)$ by
       \begin{align*}
	      h(x)=
              \begin{cases}
	             f(x)&\text{if $x\in\supp f$}\\
                     g(x)&\text{if $x\in\supp g$}\\
                     y_0 &\text{otherwise}.
	      \end{cases}
       \end{align*}
       This is well-defined. Next, $h\geqq f$ because
       \begin{align*}
	     (h\sci \supp f)(x)=&
                 \begin{cases}
		         h(x)&\text{if $x\in\supp f$}\\
                         y_0&\text{if $x\notin\supp f$}
		 \end{cases}
                 \\
                 =&
                 \begin{cases}
		        f(x)&\text{if $x\in\supp f$}\\
                        y_0&\text{if $x\notin\supp f$}
		 \end{cases}
                 \\
                 =&f(x)
       \end{align*}
       for any $x\in X$. Similarly $h\geqq g$ and so $f$ and $g$ are compatible.
\end{proof}

\begin{lem}\label{lem2_Map_glueable}
       For $f\in\Map(X,Y,y_0)$, $x\in X$ and two positive numbers $r>s>0$,
       we have $\supp(f\sci B(x,r))\supset(\supp f)\cap B(x,s)$.
       Consequently, $(\supp (f\sci B(x,r)))\cap B(x,s)=(\supp f)\cap B(x,s)$.
\end{lem}
\begin{proof}
       Take $x'\in(\supp f)\cap B(x,s)$. For any $\e>0$, there is $x''\in B(x',\e)$ such that
       $f(x'')\neq y_0$. If $\e$ is small enough, this $x''$ is in $B(x,r)$ and so 
       $(f\sci B(x,r))(x'')\neq y_0$. Since $\e$ was arbitrary, $x'\in\supp(f\sci B(x,r))$.
\end{proof}

\begin{lem}\label{lem3_map_glueable}
       Let $\Xi$ be a subset of $\Map(X,Y,y_0)$. Take $x\in X$ and two numbers $r,s$ such that
      $r>s>0$.
       If  $\Xi\sci B(x,r)$ is finite, then
       $\{(\supp f)\cap B(x,s)\mid f\in\Xi\}$ is finite. 
\end{lem}
\begin{proof}
       Clear by Lemma \ref{lem2_Map_glueable}.
\end{proof}

\begin{lem}\label{lem4_map_glueable}
      If $\Xi\subset\Map (X,Y,y_0)$ is locally finite, then $\bigcup_{f\in\Xi}\supp f$ is closed.
\end{lem}
\begin{proof}
      Take $x\in X\setminus(\bigcup_{f\in\Xi}\supp f)$.
      Since $\Xi\sci B(x,1)$ is finite, by Lemma \ref{lem3_map_glueable},
      $\{(\supp f)\cap B(x,\frac{1}{2})\mid f\in\Xi\}$ is finite.
      There is $r>0$ such that $B(x,r)\cap\supp f=\emptyset$ for any $f\in\Xi$.
\end{proof}

\begin{prop}\label{prop_map_glueable}
      $\Map(X,Y,y_0)$ is glueable.
\end{prop}
\begin{proof}
      Suppose $\Xi\subset\Map(X,Y,y_0)$ is locally finite and pairwise compatible.
      Set
         \begin{align*}
	      f_{\Xi}(x)&=
              \begin{cases}
	             f(x)&\text{if there is $f\in\Xi$ such that $x\in\supp f$}\\
                     y_0&\text{otherwise}
	      \end{cases}\\
	     &=
	    \begin{cases}
	           f(x)&\text{if there is $f\in\Xi$ such that $f(x)\neq y_0$}\\
	           y_0&\text{otherwise}.
	    \end{cases}
	 \end{align*}
      This is well-defined by Lemma \ref{lem1_map_glueable}.
      For each $f\in\Xi$ and $x\in X$,
      \begin{align*}
             (f_{\Xi}\sci\supp f)(x)=&
                             \begin{cases}
			            f_{\Xi}(x)&\text{if $x\in\supp f$}\\
                                     y_0&\text{if $x\notin\supp f$}
			     \end{cases}\\
                             =&
                             \begin{cases}
			            f(x)&\text{if $x\in\supp f$}\\
                                     y_0&\text{if $x\notin\supp f$}
			     \end{cases}\\
                             =&f(x),
      \end{align*}
       and so $f_{\Xi}\geqq f$. In other words, $f_{\Xi}$ is a majorant for $\Xi$.

       Next, we prove that $f_{\Xi}$ is the supremum for $\Xi$.
       To this end, we first claim $\supp f_{\Xi}=\bigcup_{f\in\Xi}\supp f$.
       It is clear that $\{x\in X\mid f_{\Xi}(x)\neq y_0\}\subset\bigcup_{f\in\Xi}\supp f$ because
       if $f_{\Xi}(x)\neq y_0$, then there is $f\in\Xi$ such that $f(x)\neq y_0$.
       Together with Lemma \ref{lem4_map_glueable}, we see 
       $\supp f_{\Xi}\subset\bigcup_{f\in\Xi}\supp f$. Since $f_{\Xi}$ is a majorant for
       $\Xi$, the reverse inclusion is clear, and so  $\supp f_{\Xi}=\bigcup_{f\in\Xi}\supp f$.

       To prove that $f_{\Xi}$ is the supremum, we next take an arbitrary
      majorant $g$ for $\Xi$.
       Since
       for $f\in\Xi$ and $x\in\supp f$, we have $g(x)=(g\sci\supp f)(x)=f(x)$,
       we obtain 
       \begin{align*}
	    g\sci (\bigcup_{f\in\Xi}\supp f)(x)=&
                                             \begin{cases}
					    g(x)&\text{if there is $f\in\Xi$ such that $x\in\supp f$}\\
                                            y_0&\text{otherwise}
					     \end{cases}\\
                                              =&
                                          \begin{cases}
					    f(x)&\text{if there is $f\in\Xi$ such that $x\in\supp f$}\\
                                             y_0&\text{otherwise}
					  \end{cases}\\
                                            =&f_{\Xi}(x)
       \end{align*}
       for each $x\in X$, and so $g\sci (\supp f_{\Xi})=f_{\Xi}$, namely $g\geqq f_{\Xi}$.
      We have shown $f_{\Xi}$ is the supremum for $\Xi$; $f_{\Xi}=\bigvee\Xi$.

      It remains to show that $f_{\Xi}\sci C$ is equal to $\bigvee(\Xi\sci C)$ for each
      $C\in\calC(X)$. This is the case because
      \begin{align*}
              (f_{\Xi}\sci C)=&\begin{cases}
				     f_{\Xi}(x)&\text{if $x\in C$}\\
                                     y_0 &\text{otherwise}
			       \end{cases}\\
                             =&\begin{cases}
				     f(x)&\text{if $x\in C$ and $f(x)\neq y_0$ for some $f\in\Xi$}\\
                                     y_0&\text{otherwise}
			       \end{cases}\\
                             =&\begin{cases}
				    (f\sci C)(x)&\text{if there is $f\in\Xi$ such that
                                                          $(f\sci C)(x)\neq y_0$}\\
                                    y_0&\text{otherwise}.
			       \end{cases}
      \end{align*}
\end{proof}

\subsection{Zero Element and Its Uniqueness}
\label{subsection_zero_elements}

Here we discuss ``zero elements''.
Often there is only one zero element (Lemma \ref{uniqueness_zero_element})
and this fact plays an important role later.
(See Lemma \ref{lem_thereis_R_that_decomposes}.)
\begin{defi}
      Let $\Pi$ be an abstract pattern space over a topological space $X$.
      An element $\calP\in\Pi$ such that $\supp\calP=\emptyset$ is called
      a \emph{zero element} of $\Pi$. If there is only one zero element in $\Pi$, 
      it is denoted by $0$.
\end{defi}

\begin{rem}
      Zero elements always exist. In fact, take an arbitrary element $\calP\in\Pi$.
      Then by Lemma \ref{lemma_support_calP_sci_C}, $\supp(\calP\sci\emptyset)=\emptyset$
       and so $\calP\sci\emptyset$ is a zero element.

      Zero elements need not be unique in general. For example, any set $\Pi$ is an
     abstract pattern space over any topological space $X$ by defining a cutting-off operation via
     $\calP\sci C=\calP$ for any $\calP\in\Pi$ and $C\in\calC(X)$.
     In this abstract pattern space any element $\calP\in\Pi$ is a zero element.
     The disjoint union of two abstract pattern spaces also have two zero elements.
\end{rem}

\begin{lem}\label{uniqueness_zero_element}
        If $\Pi$ is a glueable abstract pattern space over a topological space $X$,
 there is only one zero element in $\Pi$.
\end{lem}
\begin{proof}
        The subset  $\emptyset$ of $\Pi$ is locally finite and pairwise compatible.
        Set $\calP=\bigvee\emptyset$.
       By Lemma \ref{lem_support_supremum}, $\calP$ is a zero element.
       If $\calQ$ is a zero element, then since
        $\calQ$ is a majorant for $\emptyset$, we see $\calQ\geqq\calP$.
        We have $\calQ=\calQ\sci\emptyset=\calP$.
\end{proof}

We finish this subsection by proving that zero elements play no role when we take
supremum.
\begin{lem}\label{lem_supremum_Xi_and_Xi_zero}
        Let $\Pi$ be a glueable abstract pattern space over a topological space $X$.
        Take a locally finite and pairwise compatible subset $\Xi$ of $\Pi$.
        Then $\bigvee (\Xi\cup \{0\})$ exists and
        $\bigvee(\Xi\cup\{0\})=\bigvee\Xi$.
\end{lem}
\begin{proof}
        For any $\calP\in\Pi$, the abstract pattern $\calP\sci\emptyset$ is a zero element
        and by the uniqueness of zero element
        (Lemma \ref{uniqueness_zero_element}), $\calP\sci\emptyset=0$ and $\calP\geqq 0$.
        The existence and the value of the supremum will not be changed by adding
        the minimum element $0$.
\end{proof}

\section{$\Gamma$-abstract pattern spaces over $X$, or abstract pattern spaces over $(X,\Gamma)$}
\label{section_gammma-pattern_space}
Here we incorporate group actions to the theory of abstract pattern spaces.
First, we define abstract pattern spaces over $(X,\Gamma)$, or $\Gamma$-abstract pattern spaces over $X$,
where $X$ is a topological space and a group $\Gamma$ acts on $X$ by homeomorphisms.
We require there is an action of the group $\Gamma$ on such an abstract pattern space and
the cutting-off operation is
equivariant.
In Subsection \ref{subsection_local_derivable} we define local derivation by using
the structure of $\Gamma$-abstract pattern spaces. There we show several maps in aperiodic order
send an abstract pattern $\calP$ to one which is mutually locally derivable (MLD)
with $\calP$; we solve the first question in Introduction affirmatively.
\subsection{Definition and Examples}\label{subsection_def_Gamma-pat_sp}
\begin{setting}
       \emph{In this subsection, unless otherwise stated, $X$ is a topological space,
       $\Gamma$ is a group that acts on $X$ as homeomorphisms, and $\Pi$ is a 
       abstract pattern space over $X$.}
\end{setting}

In this subsection we define abstract pattern spaces over $(X,\Gamma)$, or $\Gamma$-abstract pattern spaces
over $X$. We then study relations between the group action and the notions appeared in
Section \ref{section_pattern_space}. We also give examples of abstract pattern spaces
over $(X,\Gamma)$.

\begin{defi}\label{def_gamma-pattern_space}
       Given a group action $\Gamma\curvearrowright\Pi$ such that 
       for each $\calP\in\Pi, C\in\calC(X)$ and $\gamma\in\Gamma$,
       we have $(\gamma\calP)\sci (\gamma C)=\gamma(\calP\sci C)$,
       that is, the cutting-off operation
       is equivariant, we say
       $\Pi$ is a \emph{$\Gamma$-abstract pattern space} or a \emph{abstract pattern space over $(X,\Gamma)$}.
 
        For an abstract pattern space $\Pi$ over $(X,\Gamma)$, a nonempty subset $\Sigma$ of $\Pi$
       such that $\calP\in\Sigma$ and $\gamma\in\Gamma$ imply $\gamma\calP\in\Sigma$
      is called a \emph{subshift} of $\Pi$.
\end{defi}
Examples are given after the following few
 lemmas. First we describe the relation  among the
group action $\Gamma\curvearrowright\Pi$, the supports and the order $\geqq$.

\begin{lem}\label{lem1_before_examples_equiv_pat_sp}
      Let $\Pi$ be an abstract pattern space over $(X,\Gamma)$. For $\calP,\calQ\in\Pi$ and 
      $\gamma\in\Gamma$, the following statements hold:
      \begin{enumerate}
       \item $\gamma\supp\calP=\supp(\gamma\calP)$.
	\item If $\calP\geqq\calQ$, then $\gamma\calP\geqq\gamma\calQ$.
      \end{enumerate}
\end{lem}

Next, we prove two lemmas that are concerned with the construction of new
abstract pattern spaces over $(X,\Gamma)$ from existing abstract pattern spaces over $(X,\Gamma)$. The first way of construction is taking subspace.

\begin{lem}\label{lem_pattern_subspace_over_X_Gamma}
     Let $\Pi$ be an abstract pattern space over $(X,\Gamma)$.
       Suppose $\Pi'$ is an abstract pattern subspace of
      $\Pi$. If $\Pi'$ is closed under the $\Gamma$-action (that is, it is a subshift),
 then $\Pi'$ is an abstract pattern space
       over $(X,\Gamma)$.
\end{lem}

The second way is to take the product.

 \begin{lem}\label{lem_product_Gamma-pattern_space}
       Let $\Lambda$ be a set and $(\Pi_{\lambda})_{\lambda\in\Lambda}$ be a family of
       abstract pattern spaces over $(X,\Gamma)$.
       Then $\Gamma$ acts on the product space $\prod_{\lambda}\Pi_{\lambda}$ by 
       $\gamma(\calP_{\lambda})_{\lambda}=(\gamma\calP_{\lambda})_{\lambda}$ and
      by this action $\prod_{\lambda}\Pi_{\lambda}$ is an abstract pattern space over $(X,\Gamma)$.
 \end{lem}
\begin{proof}
       That $\prod\Pi_{\lambda}$ is an abstract pattern space is proved in Lemma 
       \ref{lemmma_product_pattern_space}. For $\gamma\in\Gamma, 
       (\calP_{\lambda})\in\prod\Pi_{\lambda}$ and $C\in\calC(X)$,
       $\gamma((\calP_{\lambda})_{\lambda}\sci C)=(\gamma(\calP_{\lambda})_{\lambda})\sci \gamma C$
       by a straightforward computation.
\end{proof}

\begin{defi}\label{def_product_Gamma_pattern_spd}
       The abstract pattern space $\prod\Pi_{\lambda}$ is called the \emph{product $\Gamma$-abstract pattern space}.
\end{defi}
By this construction we see the abstract pattern space $\UD^I(X)$ of uniformly
discrete multi sets
(Example \ref{ex_multi_set}) is a $\Gamma$-abstract pattern space and the space of all
Delone multi set is its subshift. We also use the structure of $\Gamma$-abstract
pattern space on the product in Proposition \ref{prop_calP_MLD_Gamma_lambda}.

The third way is to take the power set.
\begin{lem}\label{lem_power_equivariant_pat_sp}
       Let $\Pi$ be an abstract pattern space over $(X,\Gamma)$.
       Then the power abstract pattern space $2^{\Pi}$ (Definition \ref{Def_power_pattern_space})
         is an abstract pattern space over $(X,\Gamma)$ by an action 
      $\gamma\Xi=\{\gamma\calP\mid\calP\in\Xi\}$.
\end{lem}

We now collect examples of abstract pattern spaces over $(X,\Gamma)$.
\begin{ex}\label{ex_patch_Gamma_pattern_sp}
     Suppose $X$ is a metric space and the action $\Gamma\curvearrowright X$ is
      isometric.
     For $\calP\in\Patch(X)$ and $\gamma\in\Gamma$,
      set $\gamma\calP=\{\gamma T\mid T\in\calP\}$.
     This defines an action of $\Gamma$ on $\Patch(X)$ and makes $\Patch(X)$ a 
     abstract pattern space over $(X,\Gamma)$.
\end{ex}

\begin{ex}\label{2X_as_Gamma_pattern_sp}
     Let $X$ be a metric space and let a group $\Gamma$ act on $X$ as isometries.
     $2^X$ (Example \ref{example_2^X_calC(X)}) is an abstract pattern space over $(X,\Gamma)$.
    By Lemma \ref{lem_pattern_subspace_over_X_Gamma},  the spaces $\LF(X)$(Example \ref{example_LF(X)}),
     $\UD(X)$ and $\UD_r(X)$ (Example \ref{example_UD}, $r>0$) are all
     abstract pattern spaces over $(X,\Gamma)$.
\end{ex}

The following example anticipates an application of our theory to the theory of
pattern-equivariant functions, which were defined in \cite{MR1985494} and
\cite{de2006pattern}. By this group action pattern equivariance becomes equivalent to
local derivability. See Section \ref{section_application_translation_thm}.
\begin{ex}\label{ex_map_rho}
       Take a non-empty set $Y$, an element $y_0\in Y$
      and an action $\phi\colon\Gamma\curvearrowright Y$ that fixes $y_0$.
      As was mentioned before (Example \ref{example_map}), 
       $\Map (X,Y,y_0)$ is an abstract pattern space over $X$.       
       Define an action of $\Gamma$ on $\Map(X,Y,y_0)$ by
        \begin{align*}
	       (\gamma f)(x)=\phi(\gamma)(f(\gamma^{-1}x)).
	\end{align*}
        For each $f\in\Map(X,Y,y_0)$, $\gamma\in\Gamma$ and $C\in\calC(X)$,
        \begin{align*}
	      (\gamma f)\sci (\gamma C)(x)=&\begin{cases}
					          (\gamma f)(x)&\text{if $x\in\gamma C$}\\
					          y_0         &\text{otherwise}
					   \end{cases}\\
	                                    =&\begin{cases}
					           \phi(\gamma)(f(\gamma^{-1}x))&
                                                          \text{if $\gamma^{-1}x\in C$}\\
					           \phi(\gamma)y_0&\text{otherwise}
					      \end{cases}\\
	                                    =&\phi(\gamma)(f\sci C)(\gamma^{-1}x)\\
	                                   =&\gamma(f\sci C)(x),
	\end{align*}
         for each $x\in X$ and so $\Map(X,Y,y_0)$ is an abstract pattern space over $(X,\Gamma)$.
         This $\Gamma$-abstract pattern space is denoted by $\Map_{\phi}(X,Y,y_0)$.
         If $\phi$ sends every group element to the identity, we denote the corresponding
         space by $\Map(X,Y,y_0)$.
\end{ex}

\begin{ex}\label{ex_Gamma-pattern_space_measures}
       Let $X$ be a locally compact $\sigma$-compact space and let  a group $\Gamma$ act on
       $X$ as homeomorphisms.
       The dual space $C_c(X)^{*}$ with respect to the inductive limit topology
       is an abstract pattern space over $X$
       (Example \ref{ex_space_of_measures}).
       For $\varphi\in C_c(X)$ and $\gamma\in\Gamma$, set
      $(\gamma\varphi)(x)=\varphi(\gamma^{-1}x)$.
       For $\Phi\in C_c(X)^*$ and $\gamma\in\Gamma$, set
       $\gamma\Phi(\varphi)=\Phi(\gamma^{-1}\varphi)$.
       Then $C_c(X)^{*}$ is an abstract pattern space over $(X,\Gamma)$.
\end{ex}
We have introduced various examples of abstract pattern spaces over $(X,\Gamma)$. Next,
we mention three examples of subshifts.
\begin{ex}\label{ex_subshift_of_Delone_sets}
       For a (proper)
      metric space $X$,
        the set $\Del(X)$ of all Delone sets in $X$ (Example \ref{example_UD})
       is a subshift of $\UD(X)$.
\end{ex}

\begin{ex}
       For a metric space $X$, a patch $\calT\in\Patch(X)$ is called a tiling
       if $\supp\calT=X$. The space of all tilings is a subshift of $\Patch(X)$.
\end{ex}

\begin{ex}\label{ex_subshift_conti_maps}
      In example \ref{ex_map_rho}, assume $Y$ is a topological space and each
 $\phi(\gamma)$ is continuous.
      The space $C(X,Y)$ of all continuous maps  is a subshift of
      $\Map_{\phi}(X,Y,y_0)$.
\end{ex}

We defined the terms ``locally finite'', ``pairwise compatible'' and ``glueable''
for abstract pattern spaces over metric spaces in Definition \ref{def_local_finite_compatible}
and Definition \ref{def_glueable_pattern_space}.
Using these definitions we define the following concepts.

\begin{defi}
        Let $X$ be a metric space and $\Gamma$ a group which acts on
       $X$ as isometries.
        Let $\Pi$ be an abstract pattern space over $(X,\Gamma)$.
        We say $\Pi$ is a \emph{glueable abstract pattern space over $(X,\Gamma)$}
 if it is a glueable
        abstract pattern space over $X$.
        For a glueable abstract pattern space $\Pi$ over $(X,\Gamma)$,
        a subset $\Sigma$ of $\Pi$ such that
 $\bigvee\Xi\in\Sigma$
        for any pairwise compatible and locally finite subset $\Xi$ of $\Sigma$
        is said to be \emph{supremum-closed}. (Here, the supremum $\bigvee\Xi$ exists
       by assumption.)
\end{defi}
We have introduced gluing in $\Gamma$-abstract pattern spaces and it is natural to ask the relation
between the group action and the operation of taking supremum.
We show $\bigvee$ and the group action is commutative.
\begin{lem}\label{lem_sup_action_commute}
        Let $X$ be a metric space and $\Gamma$ a group which acts on $X$ as isometries.
        Let $\Pi$ be a glueable abstract pattern space over $(X,\Gamma)$.
        If $\gamma\in\Gamma$ and $\Xi\subset\Pi$ is a subset which is
        both locally finite and pairwise compatible,
        then $\gamma\Xi$ (Lemma \ref{lem_power_equivariant_pat_sp}) is both locally finite and pairwise compatible.
        In this case, we have
        \begin{align*}
	         \gamma\bigvee\Xi=\bigvee(\gamma\Xi).
	\end{align*}
\end{lem}
\begin{proof}
        If $\calP\in\Xi$ and $\calQ\in\Xi$, then there is $\calR\in\Pi$ such that
        $\calR\geqq\calP$ and $\calR\geqq\calQ$. By Lemma \ref{lem1_before_examples_equiv_pat_sp},
        we see $\gamma\calR\geqq\gamma\calP$ and $\gamma\calR\geqq\gamma\calQ$ and so
        $\gamma\calP$ and $\gamma\calQ$ are compatible.
        If $x\in X$ and $r>0$ then since $\gamma$ is an isometry, 
        $\gamma^{-1}B(x,r)=B(\gamma^{-1}x,r)$.
        By
        \begin{align*}
	      \{\gamma\calP\sci B(x,r)\mid\calP\in\Xi\}=
	        \gamma\{\calP\sci B(\gamma^{-1}x,r)\mid\calP\in\Xi\},
	\end{align*}
         we see this set is finite. We have proved $\gamma\Xi$ is both pairwise compatible
         and locally finite.

        Next, we show the latter statement.
        We use Lemma \ref{lem1_before_examples_equiv_pat_sp} several times.
        For any $\calP\in\Xi$, $\gamma\bigvee\Xi\geqq\gamma\calP$.
        This means that $\gamma\bigvee\Xi$ is a majorant for $\gamma\Xi$.
        To show this is the supremum, take a majorant $\calR$ for $\gamma\Xi$.
        Then
         $\gamma^{-1}\calR$ is a majorant for $\Xi$ and so $\gamma^{-1}\calR\geqq\bigvee\Xi$.
         We have $\calR\geqq\gamma\bigvee\Xi$, and so 
         $\gamma\bigvee\Xi$ is the supremum for $\gamma\Xi$.
\end{proof}

\subsection{Local derivability}
\label{subsection_local_derivable}
\begin{setting}
    \emph{In this subsection, $X,Y$ and $Z$ are non-empty proper metric spaces and
     $\Gamma$ is a group which acts on $X,Y$ and $Z$ as isometries.}
\end{setting}

Local derivability was defined in \cite{MR1132337} for tilings or more generally patterns
in $\Rd$.
Here we define local derivability
 for  two abstract patterns $\calP_1$ and $\calP_2$.
 Note that these $\calP_1$ and $\calP_2$ may be in different abstract pattern spaces
$\Pi_1$ and $\Pi_2$, and these $\Pi_1$ and $\Pi_2$ may be over different metric spaces
$X$ and $Y$.
However, we assume $\Pi_1$ and $\Pi_2$ are $\Gamma$-abstract pattern spaces for the same
group $\Gamma$.

 Our definition is equivalent to the original definition in \cite{MR1132337}
 (see also \cite{baake2013aperiodic}, p.133) for patterns under an assumption
 (Lemma \ref{lem_LD_equiv_with_original}).

We first prove a lemma in order to define local derivability in our setting.
Recall that in the original definition of local derivability, a pattern $\calQ$
in $\Rd$ is
locally derivable from a pattern $\calP$ in $\Rd$ if there is a compact $K\subset\Rd$
such that
$x,y\in\Rd$ and $(\calP-x)\sqcap K=(\calP-y)\sqcap K$ always imply
$(\calQ-x)\sqcap \{0\}=(\calQ-y)\sqcap \{0\}$, where the operation $\sqcap$ is defined
in Definition \ref{def_sqcap}.
We replace $\calP$ and $\calQ$ with
general abstract patterns and $x$ and $y$ with elements $\gamma,\eta\in\Gamma$.
We replace $\sqcap$ with $\sci$, and cut off in a slightly different manner.
(In general situation $\sqcap$ cannot be defined. In order to include measures,
the space of which does not admit an analogue of $\sqcap$, we replace
$\sqcap$ with $\sci$.)
Since there is no special point like $0\in\Rd$ in general metric spaces,
we choose points $x_0\in X$ and $y_0\in Y$, and show the definition is independent of
this choice (Lemma \ref{lem_local_derivability}).
The equivalence between our definition and the original one under an assumption is
proved in Lemma \ref{lem_LD_equiv_with_original} and
Corollary \ref{cor_equiv_ourLD_original}.
(These are well-known if $X=\Rd$ but we show them in a more general situation.)
Since we often assume the first condition in Assumption \ref{simplicity_assumption},
this equivalence assures that our definition is sufficient
in this article. 

\begin{lem}\label{lem_local_derivability}
        Let $\Pi_1$ be an abstract pattern space over $(X,\Gamma)$ and $\Pi_2$ an abstract pattern space over
        $(Y,\Gamma)$.
         For two abstract patterns $\calP_1\in\Pi_1$ and $\calP_2\in\Pi_2$,
         consider the following three conditions:
         \begin{enumerate}
	  \item There exist $x_{0}\in X$, $y_0\in Y$ and $R_0\geqq 0$ such that 
		if $\gamma,\eta\in\Gamma$, $R\geqq 0$ and
		\begin{align*}
		     (\gamma\calP_1)\sci B(x_0,R+R_0)=(\eta\calP_1)\sci B(x_0,R+R_0),
		\end{align*}
		then
		\begin{align*}
		     (\gamma\calP_2)\sci B(y_0,R)=(\eta\calP_2)\sci B(y_0,R).
		\end{align*}
	  \item For any $x_{1}\in X$ and  $y_1\in Y$ there exists $R_1\geqq 0$ such that 
		if $\gamma,\eta\in\Gamma$, $R\geqq 0$ and
		\begin{align*}
		     (\gamma\calP_1)\sci B(x_1,R+R_1)=(\eta\calP_1)\sci B(x_1,R+R_1),
		\end{align*}
		then
		\begin{align*}
		     (\gamma\calP_2)\sci B(y_1,R)=(\eta\calP_2)\sci B(y_1,R).
		\end{align*}
	\item For any compact $K_2\subset Y$ there exists a compact $K_1\subset X$ such
	       that $\gamma,\eta\in\Gamma$ and
	      \begin{align}
	            (\gamma\calP_1)\sci K_1=(\eta\calP_1)\sci K_1\label{eq_sci_K_1}
	      \end{align}
	      imply
	      \begin{align}
	             (\gamma\calP_2)\sci K_2=(\eta\calP_2)\sci K_2.\label{eq_sci_K_2}
	      \end{align}
	 \end{enumerate}
         Then condition 1 and 2 are always equivalent and condition 1 and 2
 always
 imply condition 3. If the action $\Gamma\curvearrowright Y$ is transitive,
      $\Pi_2$ is glueable, $X=Y$ and
     $\calP_2$ consists of bounded components
       (Definition \ref{def_bounded_coponents}), then  condition 3  implies
         condition 1 and 2.
\end{lem}
\begin{proof}
       In order to prove the equivalence of 1 and 2, it suffices to prove
     the implication 1$\Rightarrow$2.
       If we assume 1, there are $x_0,y_0$ and $R_0$ that satisfy the condition in 1.
       Take $x_1\in X$ and $y_1\in Y$ arbitrarily.
       Set $R_1=R_0+\rho_X(x_0,x_1)+\rho_Y(y_0,y_1)$, where $\rho_X,\rho_Y$ are
       the metrics for  $X$ and $Y$, respectively.
       Take $\gamma,\eta\in\Gamma$ and $R>0$ arbitrarily and suppose
       \begin{align}
	      (\gamma\calP_1)\sci B(x_1,R_1+R)=(\eta\calP_1)\sci B(x_1,R_1+R).
              \label{eq_lemma_local_derive}
       \end{align}
       Since $B(x_0,R+\rho_Y(y_0,y_1)+R_0)\subset B(x_1, R_1+R)$,
       by taking cutting-off operation for both sides of (\ref{eq_lemma_local_derive}),
       we obtain
       \begin{align*}
	     (\gamma\calP_1)\sci B(x_0, R+\rho_{Y}(y_0,y_1)+R_0)=
              (\eta\calP_1)\sci B(x_0, R+\rho_{Y}(y_0,y_1)+R_0),
       \end{align*}
       and so 
       \begin{align*}
	        (\gamma\calP_2)\sci B(y_0, R+\rho_{Y}(y_0,y_1))=
                  (\eta\calP_2)\sci B(y_0, R+\rho_{Y}(y_0,y_1)).
       \end{align*}
     By $B(y_1,R)\subset B(y_0,R+\rho_Y(y_0,y_1))$,
       \begin{align*}
	      (\gamma\calP_2)\sci B(y_1, R)=
	       (\eta\calP_2)\sci B(y_1, R).
       \end{align*}

        Next we show that condition 1 and 2 imply condition 3.
        If we assume condition 1, there exists $R_0$ as in the condition for
        some $x_0\in X$ and $y_0\in Y$.
        Take a compact $K_2\subset Y$. We can take $R>0$ such that
          $K_2\subset B(y_0,R)$. If $\gamma,\eta\in\Gamma$ and
        \eqref{eq_sci_K_1} holds for $K_1=B(x_0,R+R_0)$, then
         \eqref{eq_sci_K_2} holds.
        
        Finally we assume $\calP_2$ consists of bounded components, $\Pi_2$ is glueable,
      $X=Y$, the action $\Gamma\curvearrowright X$ is
      transitive and condition 3
       holds. We claim condition 1 holds. Since $\calP_2$ consists of bounded
      components, we can take $R_{\calP_2}>0$
 as in Definition \ref{def_bounded_coponents}.
      To prove condition 1, 
      take $x_0\in X$. Set $K_2=B(x_0,2R_{\calP_2})$. We can take a compact
       $K_1\subset X$ as in condition 3. By compactness there exist $R_0>0$ such
       that $K_1\subset B(x_0,R_0)$. To prove this $R_0$ has the desired property,
       take arbitrary $\gamma,\eta\in\Gamma$ and $R\geqq 0$, and assume
 \begin{align*}
      (\gamma\calP_1)\sci B(x_0,R+R_0)=(\eta\calP_1)\sci B(x_0,R+R_0).
 \end{align*}
            Since $B(x_0,R)$ is compact, we can take its finite subset $F$ 
           such that $B(x_0,R)\subset\bigcup_{y\in F}B(y,R_{\calP_2})$.
           By transitivity, for each $y\in F$,
         we can take $\gamma_y\in\Gamma$ such that $\gamma_yx_0=y$.
          Moreover, we have $\rho(\gamma^{-1}_yx_0,x_0)=\rho(y,x_0)\leqq R$,
          and so
          \begin{align*}
	        (\gamma^{-1}_y\gamma\calP_1)\sci B(x_0,R_0)=(\gamma^{-1}_y\eta\calP_1)
	       \sci B(x_0,R_0),
	  \end{align*}
         which implies that
         \begin{align*}
	       (\gamma_y^{-1}\gamma\calP_2)\sci B(x_0,2R_{\calP_2})
	  =(\gamma^{-1}_y\eta\calP_2)\sci B(x_0,2R_{\calP_2}),
	 \end{align*}
         and
          \begin{align*}
	       (\gamma\calP_2)\sci B(y,2R_{\calP_2})=(\eta\calP_2)\sci B(y,2R_{\calP_2}).
	  \end{align*}
          Now the set $\{(\gamma\calP_2)\sci B(x_0,R)\sci B(y,2R_{\calP_2})\mid y\in F\}$
         is locally finite and pairwise compatible, and since $\Pi_2$ is glueable,
         this set admits a supremum $\calQ$ and $\calQ\leqq(\gamma\calP_2)\sci B(x_0,R)$.
       By the definition of $R_{\calP_2}$,
       $\supp\calQ=\supp((\gamma\calP_2)\sci B(x_0,R))$ and so
      $(\gamma\calP_2)\sci B(x_0,R)=\calQ$.
       The same argument holds for $\eta$. Hence
        \begin{align*}
	      (\gamma\calP_2)\sci B(x_0,R)&=
	     \bigvee\{(\gamma\calP_2)\sci B(x_0,R)\sci B(y,2R_{\calP_2})\mid y\in F\}\\
	    &= \bigvee\{(\eta\calP_2)\sci B(x_0,R)\sci B(y,2R_{\calP_2})\mid y\in F\}\\
	   &=(\eta\calP_2)\sci B(x_0,R),
	\end{align*}
       which completes the proof.
\end{proof}

\begin{rem}
       The proof of the implication from condition 3 to
    condition 1 can be modified so
       that we can replace the assumption $X=Y$ with a relation between the metrics
       $\rho_X$ and $\rho_Y$ of $X$ and $Y$, respectively. For example, if there exist
       $L>0$, $x_0\in X$ and $y_0\in Y$ such that
       \begin{align}
	   \rho_X(\gamma x_0,\eta x_0)\leqq\rho_{Y}(\gamma y_0,\eta y_0)+L
	\label{ineq_for_equiv_LD}
       \end{align}
       for each $\gamma,\eta\in\Gamma$, then we can prove condition 1 from
       condition 3.
    As an example, take a closed subgroup $\Gamma$ of the Euclidean group $\Ed$
that includes
   $\Rd$.
     If $X$ and $Y$ are either $\Rd$ or $\Gamma$, on which $\Gamma$ acts transitively,
      then the inequality \eqref{ineq_for_equiv_LD} holds and so we can prove the
 condition 1 from condition 3.
\end{rem}

\begin{defi}\label{def_local_derive}
        Let $\Pi_1$ be an abstract pattern space over $(X,\Gamma)$ and $\Pi_2$
        be an abstract pattern space over $(Y,\Gamma)$.
        If $\calP_1\in\Pi_1$ and $\calP_2\in\Pi_2$ satisfy
        condition 1 (and 2) in
        Lemma \ref{lem_local_derivability}, then we say $\calP_2$ is \emph{locally derivable from $\calP_1$}
        and write $\calP_1\LD\calP_2$.       
        If both $\calP_1\LD\calP_2$ and $\calP_2\LD\calP_1$ hold,
        we say $\calP_1$ and $\calP_2$ are \emph{mutually locally derivable (MLD)}
         and write
        $\calP_1\MLD\calP_2$. 
\end{defi}

We next prove that under an mild assumption our definition of local derivability is
equivalent to the original one in \cite{MR1132337}.

\begin{lem}\label{lem_LD_equiv_with_original}
    Assume the action $\Gamma\curvearrowright X$ is
    transitive. Take $x\in X$ and two patterns $\calP_1,\calP_2\in\Pattern(X)$
    (Definition \ref{def_patterns}). 
    Assume $\sup_{T\in\calP_2}\diam T<\infty$.
    Then the following two conditions are equivalent:
    \begin{enumerate}
      \renewcommand{\labelenumi}{(\alph{enumi}).}
     \item $\calP_1\LD\calP_2$.
	\item        there exists a compact $K\subset X$ such that,
       	    if $\gamma,\eta\in\Gamma$ and
       	        $(\gamma\calP_1)\sqcap K=(\eta\calP_2)\sqcap K$,
       	  then
       	        $(\gamma\calP_2)\sqcap\{x\}=(\eta\calP_2)\sqcap\{x\}$.
    \end{enumerate}    
\end{lem}
\begin{proof}
      Take $L>\sup_{T\in\calP_2}\diam T$.
      We first assume condition (a)
      and prove condition (b).
Setting  $x_1=y_1=x$, we get some $R_1\geqq 0$ such that
condition 2 in Lemma \ref{lem_local_derivability}
      holds. If $\gamma,\eta\in\Gamma$ and
       \begin{align*}
(\gamma\calP_1)\sqcap B(x,R_1+L)=(\eta\calP_1)\sqcap B(x,R_1+L),
       \end{align*}
     then
      \begin{align*}
  (\gamma\calP_1)\sci B(x,R_1+L)=(\eta\calP_1)\sci B(x,R_1+L)       
      \end{align*}
 and
 \begin{align*}
  (\gamma\calP_2)\sci B(x,L)=(\eta\calP_2)\sci B(x,L).
 \end{align*}
 By the definition of $L$, we have $(\gamma\calP_2)\sqcap \{x\}=(\eta\calP_2)\sqcap \{x\}$.

       Next, we assume condition (b)
 and prove condition (a).
          There exists $K$ as in condition (b). Since $K$ is compact, we can take
          $R_0\geqq 0$ such that $K\subset B(x,R_0)$. To prove condition (a), take
            $\gamma,\eta\in\Gamma$ and $R\geqq 0$ and assume
          $(\gamma\calP_1)\sci B(x,R_0+L+R)=(\eta\calP_1)\sci B(x,R_0+L+R)$.
         Since the action is transitive, for each $y\in B(x,R)$ there exists $\xi\in\Gamma$
 such that $\xi x=y$. By $B(x,R_0+L)\subset B(\xi^{-1}x,R_0+L+R)$, we have
         \begin{align*}
	       (\xi^{-1}\gamma\calP_1)\sci B(x,R_0+L)=(\xi^{-1}\eta\calP_1)\sci B(x,R_0+L),
	 \end{align*}
             and
           \begin{align*}
	         (\xi^{-1}\gamma\calP_1)\sqcap K=(\xi^{-1}\eta\calP_1)\sqcap K.
	   \end{align*}
           By the definition of $K$, we have
         \begin{align*}
	        (\xi^{-1}\gamma\calP_2)\sqcap \{x\}=(\xi^{-1}\eta\calP_2)\sqcap \{x\},
	 \end{align*}
         and $(\gamma\calP_2)\sqcap \{y\}=(\eta\calP_2)\sqcap \{y\}$.
         Since $y$ is arbitrary, we have
         $(\gamma\calP_2)\sqcap B(x,R)=(\eta\calP_2)\sqcap B(x,R)$, and
         $(\gamma\calP_2)\sci B(x,R)=(\eta\calP_2)\sci B(x,R)$.
\end{proof}

 \begin{cor}\label{cor_equiv_ourLD_original}
Suppose the action $\Gamma\curvearrowright X$ is transitive and
   $\calP_1$ and $\calP_2$ are
  patches in $X$. If $\calP_2$ consists of bounded components, 
  then
  our definition  of $\calP_1\LD\calP_2$ coincides with the original definition
  (condition (b) in Lemma \ref{lem_LD_equiv_with_original}).
 \end{cor}
\begin{proof}
       A patch $\calP$ consists of bounded components if and only if
 $\sup_{T\in\calP}\diam T<\infty$.
\end{proof}


The following two lemmas are easy to prove. First we show $\MLD$ is an equivalence
relation.
\begin{lem}
        \begin{enumerate}
	 \item 	Let $\calP$ be an abstract pattern in an abstract pattern space over $(X,\Gamma)$.
               Then $\calP\MLD\calP$.
        \item  Let $\calP,\calQ$ and $\calR$ be abstract patterns in abstract pattern spaces
               over $(X,\Gamma),(Y,\Gamma)$, and $(Z,\Gamma)$, respectively.
                If $\calP\LD\calQ$ and $\calQ\LD\calR$, then $\calP\LD\calR$.
                Consequently, if $\calP\MLD\calQ$ and $\calQ\MLD\calR$, then
	      $\calP\MLD\calR$.
	\end{enumerate} 
\end{lem}

Next we investigate a relation between $\LD$ and the group action
$\Gamma\curvearrowright\Pi$.
\begin{lem}
       Let $\Pi_1$ be an abstract pattern space over $(X,\Gamma)$ and $\Pi_2$ be an abstract pattern space
        over $(Y,\Gamma)$.
       Take two abstract patterns $\calP_1\in\Pi_1$ and $\calP_2\in\Pi_2$
       and suppose $\calP_1\LD\calP_2$.
       Then for any $\gamma\in\Gamma$, we have $\gamma\calP_1\LD\gamma\calP_2$.
\end{lem}

We use the following notion 
 in Section \ref{setting_of_main_thm}. This comprises one of Assumption
 \ref{simplicity_assumption}, which we noted
 in Introduction, to restrict the object of study to interesting abstract patterns.

\begin{defi}\label{def_Delone_deriving}
     Let $\Pi$ be an abstract pattern space over $(X,\Gamma)$.   
      $\calP\in\Pi$ is said to be \emph{Delone-deriving} if
       there is a Delone set $D$ in $X$ such that $\calP\LD D$.
\end{defi}

\begin{rem}
      Delone sets are Delone-deriving.
      If a tiling consists of finitely many types of tiles up to $\Gamma$
      and each tile $T$ admits a fixed point of its symmetry group $\Sym_{\Gamma}T$,
      then the tiling is Delone-deriving.
      On the other hand, constant functions are not Delone-deriving.
      Note that the symmetry group of a Delone set is discrete whereas the symmetry
      group of a constant function is not.
      If an abstract pattern is Delone-deriving, then it is ``discrete'' in a sense.

 It is worth noting that if $\mu\in C_c(X)^*$ and $|\mu|$ is its total variation, we
 have $\mu\LD |\mu|\LD\supp|\mu|$, where $\supp|\mu|$ coincides with the usual support of
 the positive measure $|\mu|$. In particular, if $\mu=\sum_{x\in D}w(x)\delta(x)$ for some
 function $w$ and a Delone $D\subset X$, then $\mu\LD D$, and $\mu$ is Delone-deriving.
\end{rem}

We finish this subsection by showing several canonical maps in aperiodic order
send an abstract pattern $\calP$ to one which is MLD with $\calP$.

It is common to convert a Delone set into a measure consisting of Dirac measures on each
point (\cite[Example 8.6]{baake2013aperiodic}). We show these abstract patterns are MLD.
\begin{prop}\label{prop_MLD_D_sum_of_dirac_delta}
       Let $X$ be a locally compact proper
       metric space on which a group $\Gamma$ acts as
       isometries. Let $D$ be a uniformly discrete subset of $X$ and
       set $\mu=\sum_{x\in D}\delta_x$, the sum of Dirac measures with
      respect to the vague topology. If we regard $D$ as an abstract pattern of
       $\UD(X)$ (Example \ref{2X_as_Gamma_pattern_sp}) 
       and $\mu$ an abstract pattern of $C_c(X)^{*}$
        (Example \ref{ex_Gamma-pattern_space_measures}),
        we have the following:
       \begin{enumerate}
	\item $\mu\sci C=\sum_{x\in D\cap C}\delta_x$ for each $C\in\calC(X)$, 
        \item $\gamma\mu=\sum_{x\in\gamma D}\delta_x$, and
        \item $\mu\MLD D$.
        \end{enumerate}
\end{prop}
\begin{proof}
      The first two are clear by definition and the third condition follows from
       the first two conditions.
\end{proof}

It is common to identify a continuous bounded function $f$
on a locally compact abelian group
and $fd\mu$, $\mu$ being a Haar measure. See for example
\cite[Proposition 4.10.5, Lemma 5.4.6]{baake2017aperiodic}.
We show these are MLD.
\begin{prop}\label{prop_MLD_f_fdmu}
    Let $\Gamma$ be a $\sigma$-compact
    locally compact abelian group and $\mu$ its Haar measure.
    Let $f$ be a complex valued continuous bounded function on $\Gamma$.
    If we regard  $f$ as an abstract pattern in $\Map(\Gamma,\mathbb{C},0)$
     (Example \ref{ex_map_rho}) and
      $fd\mu$ as an element of $C_c(\Gamma)^*$
       (Example \ref{ex_Gamma-pattern_space_measures}) that sends 
      $\varphi\in C_c(\Gamma)$ to $\int\varphi fd\mu$,
     we have $f\MLD fd\mu$.
\end{prop}
\begin{proof}
     Take $R>0$ and $s,t\in\Gamma$ and assume
      \begin{align}
            (f-s)\sci B(e,R)=(f-t)\sci B(e,R).
             \label{eq_f_fdmu_MLD}
      \end{align}
      Here, $f-t$ and $f-s$ denote the image of $f$ by the group action.
     For each $\varphi\in C_c(\Gamma)$, the image by $(fd\mu-s)\sci B(e,R)$ is
     $\int_{B(e,R)}\varphi(x)f(x+s)d\mu$ and the image by $(fd\mu-t)\sci B(e,R)$ is
     $\int_{B(e,R)}\varphi(x)f(x+t)d\mu$.
     By (\ref{eq_f_fdmu_MLD}), for each $x\in B(e,R)$,
     \begin{align*}
           f(x+t)=(f\sci B(t,R))(x+t)=((f-t)\sci B(0,R))(x)=((f-s)\sci B(e,R))(x)=f(x+s),
     \end{align*}
     and so the images of $\varphi$ by $(fd\mu-s)\sci B(e,R)$ and $(fd\mu-t)\sci B(e,R)$
     are the same, and so these two maps are the same.

    Conversely, suppose $R>0$, $s,t\in\Gamma$ and
    \begin{align*}
             (fd\mu-s)\sci B(e,R+1)=(fd\mu-t)\sci B(e,R+1).
    \end{align*}
     For any $\varphi\in C_c(\Gamma)$ with $\supp\varphi\subset B(e,R+1)$,
     we have 
     \begin{align*}
           \int\varphi(x)f(x+s)d\mu(x)=\int\varphi(x)f(x+t)d\mu(x),
     \end{align*}
      and so for any $x\in B(e,R)$, we have $f(x+s)=f(x+t)$ and
      \begin{align*}
          (f-s)\sci B(e,R)=(f-t)\sci B(e,R).
      \end{align*}
\end{proof}

\emph{For the rest of this subsection,
 $(\Rd,\rho)$ is the Euclidean space with the Euclidean metric 
   and $D$ is a Delone subset (Example \ref{example_UD}) of $\Rd$
which is $R$-relatively dense  and $r$-uniformly discrete for some $R,r>0$.}

It is sometimes useful to convert $D$ in $\Rd$ into a tiling.
This is done by constructing Voronoi cells and Voronoi tilings
\cite{MR1340198}.
The set $V_x$ below (or its closure) is called the Voronoi cell of $D$ at $x$.
The set of all the Voronoi cells $V_x$, $x\in D$, forms a tiling called
Voronoi tiling or Voronoi tessellation, but
the original Delone set $D$ is not necessarily locally derivable from the tiling.
For example, consider the Delone set
$D=\{a+n\mid a\in\{\frac{1}{5},-\frac{1}{5}\},n\in\mathbb{Z}\}$ in $\mathbb{R}$.
The set of all $V_x$'s form a tiling
$\{(0,\frac{1}{2})+n\mid n\in\frac{1}{2}\mathbb{Z}\}$,
but the symmetry group of the tiling is $\frac{1}{2}\mathbb{Z}$, which is strictly larger
than the symmetry group $\mathbb{Z}$ of the original $D$.
The symmetry group is preserved under MLD, and so $D$ and this tiling are not MLD.
We circumvent this problem by considering punctured Voronoi cells $U_x$.

Although the construction is well-known, we do not omit it and prove MLD
with or without rotation.
\begin{defi}
      For each $x\in D$, we denote by $V_x$ the set
      \begin{align*}
            V_x=\{y\in \Rd\mid\text{$\rho(x,y)<\rho(x',y)$ for
                         any $x'\in D\setminus\{x\}$.}\}
      \end{align*}
\end{defi}

\begin{lem}\label{lem1_voronoi}
       For each $x\in D$, $V_x$ is nonempty and $V_x\subset B(x,R)^{\circ}$.
        Moreover,
        \begin{align}
	      V_x=\{y\in B(x,R)^{\circ}\mid\text{
                        $\rho(x,y)<\rho(x',y)$ for each
                         $x'\in D'$}\}
                 \label{eq_voronoi_lemma}
	\end{align}
         for each $D'$ with  $D\setminus\{x\}\cap B(x,2R)\subset D'
                   \subset D\setminus\{x\}$.
        In particular $V_x$ is open for each $x\in D$.
\end{lem}
\begin{proof}
         Since $x\in V_x$,
          $V_x\neq\emptyset$.
         If $y\in\Rd\setminus B(x,R)^{\circ}$, then
          since there is $x'\in D\cap B(y,R)^{\circ}$,
         we have $\rho(x',y)<R\leqq\rho(x,y)$ and so
         $y\notin V_x$.

         Assume $y\in B(x,R)^{\circ}$ and $\rho(x,y)<\rho(x',y)$ for each
         $x'\in (D\setminus\{x\})\cap B(x,2R)$. If $x'\in D\setminus\{x\}$ and
         $\rho(x,x')>2R$, then
         $\rho(x',y)\geqq\rho(x,x')-\rho(x,y)>R>\rho(x,y)$ and
         so $y\in V_x$.
         This observation shows the equality (\ref{eq_voronoi_lemma}).
\end{proof}

\begin{defi}
        For each $x\in D$, set
	      $U_x=V_x\setminus\{x\}.$
         Set $\calT=\{U_x\mid x\in D\}$.
\end{defi}

\begin{lem}
        $\calT$ is a tiling of $\Rd$.
\end{lem}
\begin{proof}
        By Lemma \ref{lem1_voronoi}, $U_x$ is open, bounded and nonempty.
        By definition of $V_x$, if $x\neq x'$ we have $U_x\cap U_{x'}=\emptyset$.
        Next we take $y\in\Rd$ and show that
       there is $x\in D$ such that $y\in\overline{U_x}$.
        To this purpose we may assume that $y\neq x$ for any $x\in D$.
        Since $\{x\in D\mid \rho(x,y)<R\}$ is finite and nonempty,
        $F=\{x\in D\mid \text{$\rho(x,y)\leqq\rho(x',y)$ for any $x'\in D$}\}$
        is nonempty and finite. Take $x\in F$.
        For each $t\in (0,1)$, set $y_t=tx+(1-t)y$.
        Then $\rho(x,y_t)=\|(1-t)(y-x)\|$.
         If $x'\in D$ and $\{y-x,y-x'\}$ is linearly independent,
         we have
         \begin{align*}
	        \rho(x',y_t)=\|(1-t)y+tx-x'\|>\|y-x'\|-t\|y-x\|\geqq(1-t)\|y-x\|
               =\rho(x,y_t).
	 \end{align*}
          If $x'\in D\setminus\{x\}$ and $\{y-x,y-x'\}$ is linearly dependent, then 
          there is $\lambda\in\mathbb{R}$ such that
          $x'-y=\lambda(x-y)$.
         Since $\lambda>1$ or $\lambda\leqq -1$, we see $\rho(y_t,x)<\rho(y_t,x')$.
         By these observations we see $y_t\in V_x$, and so
          $y\in\overline{V_x}=\overline{U_x}$.
\end{proof} 

\begin{rem}
       There is $s>0$ such that $B(x,s)\subset U_x\cup\{x\}$.
       Conversely, if $y\in\Rd\setminus U_x$ and there is $s>0$ such that
       $B(y,s)\subset U_x\cup\{y\}$, then $x=y$.
        Thus if $x,y\in D$, $\gamma,\eta\in\Gamma$ and $\gamma U_x=\eta U_y$, then
       $\gamma x=\eta y$.
\end{rem}

\begin{prop}\label{prop_MLD_voronoi}
      Let $\Gamma$ be a closed subgroup of $\Ed$.
      If we regard $D$ as an element of $\UD(\Rd)$, which is an abstract pattern space over
        $(\Rd,\Gamma)$, and $\calT$ as an element of $\Patch(\Rd)$, which is also a
        abstract pattern space over $(\Rd,\Gamma)$, we have $D\MLD\calT$.
\end{prop}
\begin{proof}
Take $L>0$ and $\gamma,\eta\in\Gamma$ and assume
\begin{align}
        (\gamma D)\cap B(0,L+2R)=(\eta D)\cap B(0,L+2R).\label{eq1_D_voronoi_MLD}
\end{align}
Suppose $x\in D$ and $\gamma U_x\subset B(0,L)$.
Since $\gamma x\in B(0,L)$, by (\ref{eq1_D_voronoi_MLD}), we see $\gamma x\in \eta D$
and $y=\eta^{-1}\gamma x\in D$.
By setting $D'=(D\setminus\{x\})\cap B(\gamma^{-1}0,L+2R)$ in Lemma \ref{lem1_voronoi},
we have
\begin{align*}
       \gamma U_x&=
          \gamma\{z\in B(x,R)^{\circ}\mid\text{$\rho(x,z)<\rho(x',z)$ for any
                       $x'\in (D\setminus\{x\})\cap B(\gamma^{-1}0,L+2R)$}\}\\
                 &=
         \{z\in B(\gamma x,R)^{\circ}\mid
                \text{$\rho(\gamma x,z)<\rho(x',z)$ for any
                   $x'\in'(\gamma D)\cap B(0,L+2R)\setminus\{\gamma x\}$ }\}\\         
                &=
         \{z\in B(\eta y,R)^{\circ}\mid
                 \text{$\rho(\eta y,z)<\rho(x',z)$ for any
                  $x'\in (\eta D)\cap B(0,L+2R)\setminus\{\eta y\})$}\}\\
               &=\eta U_y,
\end{align*}
and so $\gamma U_x\in\eta\calT$.
We have shown $(\gamma\calT)\sci B(0,L)\subset\eta\calT$ and by symmetry 
this implies that
 $(\gamma\calT)\sci B(0,L)=(\eta\calT)\sci B(0,L)$.

Conversely, assume $L>0$, $\eta,\gamma\in\Gamma$ and
\begin{align}
        (\gamma \calT)\sci B(0,L+R)=(\eta\calT)\sci B(0,L+R).
        \label{eq2_voronoi_MLD}
\end{align}
If $x\in D$ and $\gamma x\in B(0,L)$, then 
 $\gamma U_x\subset B(0, L+R)$ and so by (\ref{eq2_voronoi_MLD}) we have
$\gamma U_x\in(\eta\calT)\sci B(0,L+R)$.
There is $y\in D$ such that $\gamma U_x=\eta U_y$, and so $\gamma x=\eta y\in\eta D$.
We have shown $(\gamma D)\cap B(0,L)\subset\eta D$ and by symmetry
we obtain $(\gamma D)\cap B(0,L)=(\eta D)\cap B(0,L)$.
\end{proof}

\section{Translation theorem for certain abstract patterns}
\label{section_trans_thm}
\begin{setting}\label{setting_of_main_thm}
        \emph{In this section $X=\Rd$  and $\Gamma$ is a closed subgroup of $\Ed$ that contains
        $\Rd$. $\Pi$,
        $\Pi_1$ and $\Pi_2$ are glueable abstract pattern spaces over $(\Rd,\Gamma)$.
         }

         \emph{Note that we endow $\Gamma$ a metric
         $\rho_{\Gamma}$ given in Notation \ref{notation}.}
\end{setting}

In this section we prove Theorem \ref{translation_thm}, which answers the second
question in Introduction, Problem \ref{problem}.
The first three subsections are preliminaries for
the proof.

\subsection{Decomposition of Abstract Patterns by Delone Sets}
\label{subsection_decomposition}


To explain the idea of this subsection, consider a tiling $\calT$ in $\Rd$, where we only
consider translations. Assume the diameters of tiles is bounded from above.
Suppose
 we pick one point $x_T$ from each $T\in\calT$, in such a way that if
 $S,T\in\calT$ are translationally equivalent, then $x_T$ and $x_S$ are also
 translationally equivalent by the same vector.
Then the set $D=\{x_T\mid T\in\calT\}$
is a Delone set that is locally derivable from $\calT$. Since the diameters of tiles
in $\calT$ is bounded, if $R>0$ is large enough we have
\begin{align*}
     \calT=\bigcup_{x\in D}\calP\sci B(x,R)=\bigvee\{\calP\sci B(x,R)\mid x\in D\}.
\end{align*}
In this way we can ``decompose'' $\calT$ into family of patches
$\Xi=\{\calP\sci B(x,R)\mid x\in D\}$, from which we can reconstruct $\calT$.
Each element of $\Xi$ describes the behavior of $\calT$ around $x$, and we can take
a tuple  $(\calP_{\lambda})_{\lambda}$ of patches which are located around $0\in\Rd$
and such that for each $x\in D$ there is one and only one $\calP_{\lambda}$ that is
a translate of $\calT\sci B(x,R)$. In other words, $(\calP_{\lambda})_{\lambda}$ is the
tuple of all possible behaviors of $\calT$ around each $x\in D$.
We can reconstruct $\calT$ from ``the tuple of components'',
$(\calP_{\lambda})_{\lambda}$, and the plan, that is, the information of
``where translates of each $\calP_{\lambda}$ appears'', just as a machine or a building
is constructed from their components and plans.
We show the original $\calT$ and its plan
are MLD (Proposition \ref{prop_calP_MLD_Gamma_lambda}).


\begin{defi}\label{def_D_R_decomposes_calP}
      Take an abstract pattern $\calP\in\Pi$.
      We say a pair $(D,R)$ of a Delone set in $X$ and a positive number $R>0$
 \emph{decomposes
      $\calP$} if the following three conditions are satisfied:
      \begin{enumerate}
       \item $\calP\LD D$,
       \item $\calP=\bigvee\{\calP\sci B(x,R)\mid x\in D\}$, and
       \item $\sup_{x\in D}\card(\Sym_{\Gamma_x}\calP\sci B(x,R))$ is finite.
      \end{enumerate}
\end{defi}

In this definition, the third condition is a technical one that only arises when
we consider $\Od$-actions.
We first investigate a
relation between decomposition by a Delone set and a positive number,
and the group action $\Gamma\curvearrowright\Pi$.
\begin{lem}
      If $(D,R_0)$ decomposes $\calP$ and $\gamma\in\Gamma$, then $(\gamma D,R_0)$ decomposes
      $\gamma\calP$.
\end{lem}

\emph{For the rest of this subsection 
        $\calP$ is an element of $\Pi$,
      $D$ a Delone set in $X$ and $R_0$  a positive real number
       and we assume that $(D,R_0)$ decomposes $\calP$.}
We will use the following lemma to define tuple of components and plan.

\begin{lem}\label{lem_existence_components}
        There exist a set $\Lambda$ and 
         $\calP_{\lambda}\in\Pi$ for each $\lambda\in\Lambda$
        such that
         \begin{enumerate}
	  \item for each $\lambda\in\Lambda$, we have $\supp\calP_{\lambda}\subset B(0,R_0)$, and
          \item for each $x\in D$ there are a unique $\lambda_x\in\Lambda$ and $\gamma\in\Gamma$
                such that $\calP\sci B(x,R_0)=\gamma\calP_{\lambda_x}$ and $x=\gamma 0$.
	 \end{enumerate}
\end{lem}
\begin{proof}
      Define an equivalence relation $\sim$ on $D$ as follows:
      we have $x\sim y$ if there is $\gamma\in\Gamma$ such that (1) $\gamma x=y$, and (2)
      $\gamma(\calP\sci B(x,R_0))=\calP\sci B(y,R_0)$.
       Then by taking one point from each equivalence class for $\sim$, we obtain a set
       $\Lambda$.

      For each $x\in\Lambda$, take an element $\gamma_x\in\Gamma$ such that $\gamma_x 0=x$.
      Set $\calP_x=\gamma_x^{-1}(\calP\sci B(x,R_0))$; then $\Lambda$ and $\calP_x,x\in\Lambda$,
      satisfy the conditions.
\end{proof}

\begin{rem}
       By the second condition of Lemma \ref{lem_existence_components},
       we see $\Sym_{\Gamma_{0}}\calP_{\lambda_x}$ is conjugate
       to $\Sym_{\Gamma_x}\calP\sci B(x,R_0)$.
      In particular, $\card\Sym_{\Gamma_{0}}\calP_{\lambda}$,
       where $\lambda\in\Lambda$, is bounded from above.
\end{rem}

\begin{defi}\label{def_tuple_ingredients}
         Any tuple of abstract patterns $(\calP_{\lambda})_{\lambda\in\Lambda}$ 
         which satisfies the conditions in Lemma \ref{lem_existence_components}
         is
         called the \emph{tuple of components} for $\calP$ with respect to $(D,R_0)$.
         For each $\lambda\in\Lambda$, set
         \begin{align*}
	      P_{\lambda}= P_{\lambda}(\calP,D,R_0,(\calP_{\lambda})_{\lambda})
                  =\{\gamma\in\Gamma\mid\text{$\gamma 0\in D$ and 
                   $\calP\sci B(\gamma 0,R_0)=\gamma\calP_{\lambda}$}\}
	 \end{align*}       
         and call the tuple $(P_{\lambda})_{\lambda}$ the \emph{plan}
 for $\calP$ with respect to
        $(D,R_0,(\calP_{\lambda}))$.
\end{defi}

\begin{ex}
      Define two labeled tiles $I_W$ and $I_B$ in $\Rd$ via $I_W=([0,1]^d,W)$ and
      $I_B=([0,1]^d,B)$. These are ``black tile'' and ``white tile'' and we can
     consider a (labeled) tiling $\calT$ in $\Rd$ like a checkerboard, that is,
      the set of all tiles $I_B+(z_1,z_2,\ldots ,z_d)$ with $z_j\in\mathbb{Z}$ and
      $z_1+z_2+\cdots +z_d\in 2\mathbb{Z}$  and
 $I_W+(z_1,z_2,\ldots ,z_d)$ with $z_j\in\mathbb{Z}$ and
      $z_1+z_2+\cdots +z_d\in 2\mathbb{Z}+1$.
            A Delone set $\mathbb{Z}^d$ is locally derivable from this tiling $\calT$.
         For any large $R>0$, the tuple of components is
        two patches $\calP_{R,B}$ and $\calP_{R,W}$, where the former is nothing but
        $\calT\sci B(0,R)$ and the latter is obtained by reversing colors of  tiles
     in the former.
        The plan for these tuple of components is
        $\{(z_1,z_2,\ldots ,z_d)\in\mathbb{Z}^d\mid \sum z_j\in 2\mathbb{Z}\}$ and
         $\{(z_1,z_2,\ldots ,z_d)\in\mathbb{Z}^d\mid \sum z_j\in 2\mathbb{Z}+1\}$.
       (Here we only consider translations, but if $\Gamma$ is larger than $\Rd$
      the plan becomes bigger.)
\end{ex}

The following lemma on a relation among the group action, tuple of components
and plan is easy to prove.
\begin{lem}
      Let $(\calP_{\lambda})_{\lambda\in\Lambda}$
      be a tuple of components for $\calP$ with respect to $(D,R_0)$.
      Let $(P_{\lambda})_{\lambda\in\Lambda}$ be the plan for $\calP$ with
      respect to $(D,R_0,(\calP_{\lambda}))_{\lambda}$.
For any $\gamma\in\Gamma$, $(\calP_{\lambda})_{\lambda}$ is
     a tuple of components for $\gamma\calP$ with respect to $(D,R_0)$ and
      $(\gamma P_{\lambda})_{\lambda}$ is the plan for $\gamma\calP$ with respect to
      $(D,R_0,(\calP_{\lambda}))_{\lambda}$.
\end{lem}

\begin{rem}
       Let $(\calP_{\lambda})_{\lambda\in\Lambda}$ be a tuple of components
        for $\calP$ with respect to $(D,R_0)$.
        Let $(P_{\lambda})$ be the plan for $\calP$
        with respect to $(D,R_0,(\calP_{\lambda}))$.
       Then
\begin{align*}
           \{\calP\sci B(x,R_0)\mid x\in D\}=\{\gamma\calP_{\lambda}\mid\lambda\in\Lambda, 
             \gamma\in P_{\lambda}\}.
\end{align*}
      This implies that $\calP=\bigvee\{\gamma\calP_{\lambda}\mid\lambda\in\Lambda, 
             \gamma\in P_{\lambda}\}$.
\end{rem}

Now we prove the goal of this subsection.

\begin{prop}\label{prop_calP_MLD_Gamma_lambda}
         Let $(\calP_{\lambda})_{\lambda\in\Lambda}$ be a tuple of components
         for $\calP$ with respect to $(D,R_0)$ and
         $(P_{\lambda})$ the plan for $\calP$ with respect to 
        $(D,R_0,(\calP_{\lambda}))$.
         If we regard $(P_{\lambda})$ as an abstract pattern of 
        $\prod_{\lambda\in\Lambda}2^{\Gamma}$, which is an abstract pattern space over $(\Gamma,\Gamma)$,
        (Lemma \ref{lem_product_Gamma-pattern_space}, 
         Definition \ref{def_product_Gamma_pattern_spd}, Example
         \ref{2X_as_Gamma_pattern_sp}) 
        we have
         \begin{align*}
	        \calP\MLD (P_{\lambda})_{\lambda}.
	 \end{align*}
\end{prop}
\begin{proof}
       \underline{Step 1: We prove $\calP\LD (P_{\lambda})_{\lambda}$.}
       Let $R_1>0$ be a constant for the local derivation $\calP\LD D$
       for points $0$ and $0$
        which appears in
       the definition of local derivability (Definition \ref{def_local_derive}).
       Let $L_0$ be an arbitrary positive real number.
       Set $L_1=L_0+R_0+R_1$.
       We assume $\gamma,\eta\in\Gamma$ and
       \begin{align}
	      (\gamma\calP)\sci B(0,L_1)=(\eta\calP)\sci B(0,L_1)
               \label{eq1_proof_calP_LD_Gammalambda}
       \end{align}
       and show
       \begin{align}
	      (\gamma P_{\lambda})\cap B(e,L_0)=(\eta P_{\lambda})\cap B(e,L_0)
             \label{eq2_prop_calP_MLD_Gamma_lambda}
       \end{align}
       for each $\lambda\in\Lambda$.

       Take $\lambda\in\Lambda$ and fix it.
       By (\ref{eq1_proof_calP_LD_Gammalambda}), we see
       \begin{align*}
	     (\gamma D)\cap B(0,L_0+R_0)=(\eta D)\cap B(0,L_0+R_0).
       \end{align*}
        Let $\zeta$ be an element of $P_{\lambda}$ such that
        $\gamma\zeta\in B(e,L_0)$.
        We claim that $\gamma\zeta\in\eta P_{\lambda}$.
        By the definition of the plan,
        $\zeta 0\in D$ and $\zeta\calP_{\lambda}=\calP\sci B(\zeta 0,R_0)$.
        Since $\rho(\gamma\zeta 0,0)\leqq\rho_{\Gamma}(\gamma\zeta,e)\leqq L_0$,
        $\gamma \zeta 0\in (\gamma D)\cap B(0,L_0)=(\eta D)\cap B(0,L_0)$, and so
        there is $y\in D$ such that $\eta y=\gamma\zeta 0$.
        Now
        \begin{align*}
	       \gamma\zeta\calP_{\lambda}&=(\gamma\calP)\sci B(\gamma\zeta 0,R_0)\\
	                                &=(\gamma\calP)\sci B(0,L_1)\sci B(\gamma\zeta 0,R_0)\\
                                       &=(\eta\calP)\sci B(0,L_1)\sci B(\eta y,R_0)\\
                                      &=\eta(\calP\sci B(y,R_0)).
	\end{align*}
        We have proved $\eta^{-1}\gamma\zeta 0\in D$ and 
        $\eta^{-1}\gamma\zeta \calP_{\lambda}=\calP\sci B(\eta^{-1}\gamma\zeta 0,R_0)$, and
       so $\eta^{-1}\gamma\zeta\in P_{\lambda}$, by which we proved the claim.
       Thus $(\gamma P_{\lambda})\cap B(e,L_0)\subset(\eta P_{\lambda})\cap B(e,L_0)$
       and by symmetry we have shown (\ref{eq2_prop_calP_MLD_Gamma_lambda}).

      \underline{Step 2: We prove $(P_{\lambda})_{\lambda}\LD\calP$.}
      Let $L_0>0$ be an arbitrary positive number and set $L_1=L_0+R_0+C_0$.
      Assume $\gamma,\eta\in\Gamma$ and
      \begin{align}
             (\gamma P_{\lambda})\cap B(e,L_1)=(\eta P_{\lambda})\cap B(e,L_1)
      \end{align}
      holds for each $\lambda\in\Lambda$.
       For each $\lambda\in\Lambda$ and $\xi\in P_{\lambda}$,
       if we have $(\gamma\xi\calP_{\lambda})\sci B(0,L_0)\neq 0$, then
       $B(\gamma\xi 0,R_0)\cap B(0,L_0)\neq\emptyset$. 
       This implies $\rho(\gamma\xi 0,0)\leqq L_0+R_0$ and
       $\rho_{\Gamma}(\gamma\xi,e)\leqq L_0+R_0+C_0=L_1$.
       We have the same observation if we replace $\gamma$ with $\eta$.
       Thus
       \begin{align*}
       	       \{(\gamma\xi\calP_{\lambda})\sci B(0,L_0)
                  & \mid\lambda\in\Lambda, \xi\in P_{\lambda}\}\cup\{0\}\\
         & =\{(\gamma\xi\calP_{\lambda})\sci B(0,L_0))\mid\text{$\lambda\in\Lambda$,
                        $\xi\in P_{\lambda}$ and $\gamma\xi\in B(e,L_1)$}\}\cup\{0\}\\
        &=\{(\eta\zeta\calP_{\lambda})\sci B(0,L_0)\mid\text{$\lambda\in\Lambda$,
                    $\zeta\in P_{\lambda}$ and $\eta\zeta\in B(e,L_1)$}\}\cup\{0\}\\
        &=\{(\eta\zeta\calP_{\lambda})\sci B(0,L_0)\mid\lambda\in\Lambda,\zeta\in P_{\lambda}\}
                \cup\{0\}.
       \end{align*}
       We obtain the desired result by Lemma \ref{lem_supremum_Xi_and_Xi_zero} and
       Lemma \ref{lem_sup_action_commute}:
       \begin{align*}
	     (\gamma\calP)\sci B(0,L_0)&=\bigvee 
                     \{(\gamma\xi\calP_{\lambda})\sci B(0,L_0)
                  \mid\lambda\in\Lambda, \xi\in P_{\lambda}\}\cup\{0\}\\
                 &=\bigvee \{(\eta\zeta\calP_{\lambda})\sci B(0,L_0)
                   \mid\lambda\in\Lambda, \zeta\in P_{\lambda}\}\cup\{0\}\\
                 & =(\eta\calP)\sci B(0,L_0).  
       \end{align*}
\end{proof}

\begin{rem}
      For tilings and Delone sets we have the concept of finite local complexity
      (FLC). We can generalize this concept to arbitrary abstract pattern spaces,
      by defining an abstract pattern $\calP$ has FLC if its continuous hull with respect
 to the local matching topology is compact. (Local matching topology can be defined by
 the structure of abstract pattern space. The usual ``finitely many behaviors when seen
 from a spherical window'' is not relevant for functions such as
 $\sin\colon\mathbb{R}\rightarrow\mathbb{R}$, and we define via compactness.)
      If $\calP$ has FLC,
 then the index set $\Lambda$ in Definition \ref{def_tuple_ingredients}
       is finite.
\end{rem}

\subsection{Families of building blocks and admissible digits}
In the last subsection, we studied the decomposition of abstract patterns.
Here we study construction of abstract patterns from ``building blocks''.
\label{subsection_family_build_blocks}
\begin{setting}
       \emph{In this subsection we assume, in addition to Setting \ref{setting_of_main_thm},
       that $\Sigma$ be a supremum-closed subshift inside $\Pi$.}
\end{setting}

Here we define and study ``building blocks'' and ``admissible digits''.
For example, a square $I=(0,1)\times (0,1)$ in $\mathbb{R}^2$
is a building block, in the sense
that we can juxtapose its copies to obtain a patch. But we cannot obtain a patch from
$I$ and $I+(1/2,0)$, because they overlap. So that the digit $\{(0,0),(1/2,0)\}$,
which describes the positions of these two copies, is not ``admissible'' and we should
rule it out. If the elements of digit $\{x_1,x_2\ldots\}$ are apart enough, then
we can obtain a patch $\{I+x_1,I+x_2\ldots\}$ by juxtaposing $I$.
In this case the digit $\{x_1,x_2,\ldots\}$ is admissible.
In general, we will define ``family of building blocks'', which is a family of
abstract patterns that can be juxtaposed to obtain a new abstract patterns, if the
ambient abstract pattern space is glueable. Admissible digits are possible positions
of copies of elements in family of building block by which we can juxtapose them
 without overlap.


\begin{defi}\label{def_family_build_block}
       Take a positive number $r>0$
       arbitrarily.
       A subset $\frakF\subset\Sigma$ is called a \emph{family of building blocks of $\Sigma$
       for $r$}
       if the following three conditions are satisfied:
       \begin{enumerate}
	\item $\frakF\neq\emptyset$ and $\emptyset\ne\supp\calP\subset B(0,r)$
	      for each $\calP\in\frakF$.
	\item If $\gamma,\eta\in\Gamma$, $\calP,\calQ\in\frakF$ and 
	      $\rho(\gamma 0,\eta 0)>4r$, then
	      $\gamma\calP$ and $\eta\calQ$ are compatible.
	\item If $\calP,\calQ\in\frakF$, $\gamma\in\Gamma$ and $\gamma\calP=\calQ$, then
	      $\calP=\calQ$ and $\gamma 0=0$.
       \end{enumerate}
       The elements of $\frakF$ are called \emph{building blocks} for $r$.
       If a building block $\calP$ for $r$ additionally satisfies the condition
       \begin{align*}
	      \Sym_{\Gamma}\calP=\Gamma_{0},
       \end{align*}
       then $\calP$ is called a \emph{symmetric building block} for $r$.

       Let $\frakF$ be a family of building blocks of $\Sigma$ for $r$.
       Then a tuple $(\mathcal{D}_{\calP})_{\calP\in\frakF}$ of subsets 
       $\mathcal{D}_{\calP}\subset\Gamma$ is called an \emph{admissible digit} if
 \begin{align*}
	     \text{$\calP,\calQ\in\frakF$, $\gamma\in\mathcal{D}_{\calP}$, 
	         $\eta\in\mathcal{D}_{\calQ}$ and $\rho(\gamma 0,\eta 0)\leqq 4r$}
 \end{align*}
	      imply    $\calP=\calQ$ and $\gamma\calP=\eta\calQ$.
\end{defi}

\begin{rem}
      A non-empty subset of a family of building blocks is again a family of building blocks.
\end{rem}

 We now give examples, and after that prove lemmas that will be useful later.
\begin{ex}
      Consider a pattern space $\Patch(\Rd)$ over $(\Rd,\Rd)$
      (Example \ref{example_patch}, Example \ref{ex_patch_Gamma_pattern_sp}).
      A set $\{(0,1)^d,(0,2)^d\}$ is an example of family of building blocks for
      $2\sqrt{d}$, since if any element of one patch and any element of another
     patch are disjoint, those patches are compatible (Example
       \ref{ex_patch_compatible}).
\end{ex}

\begin{ex}
      Consider a pattern space $\UD_r(\Rd)$ over $(\Rd,\Gamma)$ (Example
       \ref{2X_as_Gamma_pattern_sp}), where $\Gamma$ is as in Setting
        \ref{setting_of_main_thm}.
A one-point set $\{0\}$
 is a symmetric
        building block for $r$, since if $\rho(\gamma 0,\eta 0)>4r$, then
        $\{\gamma 0,\eta 0\}$ majorises both $\gamma \{0\}$ and $\eta \{0\}$,
       and so these two are compatible. For example, $(\mathcal{D}_{\{0\}})$, where
        $\mathcal{D}_{\{0\}}=5r\mathbb{Z}^d$, is an admissible digit.
\end{ex}
\begin{lem}\label{lem_admissible_digit}
      Let $\frakF$ be a family of building blocks for $r$ and
      $(\mathcal{D}_{\calP})_{\calP\in\frakF}$ be an admissible digit.
      Then $\{\gamma\calP\mid\calP\in\frakF, \gamma\in\mathcal{D}_{\calP}\}$
      is locally finite and pairwise compatible.
\end{lem}
\begin{proof}
Clear by definition. 
\end{proof}
Since a family of building blocks is inside a supremum-closed subshift,
there is a supremum $\bigvee\{\gamma\calP\mid\calP\in\frakF, 
\gamma\in\mathcal{D}_{\calP}\}$ inside $\Sigma$
under the same condition as in Lemma \ref{lem_admissible_digit}.

We finish this subsection by proving two lemmas which will be useful in 
Subsection \ref{subsection_proof_translation_thm}, when we prove
Theorem \ref{translation_thm}.

\begin{lem}\label{lem_first_two_aixon_of_fam_build_block}
      Let $\frakF$ be a family of building blocks for $r$.
      Take a real number $r'>2r$ arbitrarily.
      Let $(\mathcal{D}_{\calP}^{\lambda})_{\calP\in\frakF}$ be an admissible digit
      for each $\lambda$,
      where $\lambda$ belongs to an index set $\Lambda$,
      such that
      \begin{enumerate}
       \item for each $\lambda\in\Lambda$, we have
	     $\bigcup_{\calP\in\frakF}\mathcal{D}_{\calP}^{\lambda}\neq\emptyset$, and
	\item for each $\lambda$ and $\calP$, any element $\gamma\in\mathcal{D}_{\calP}^{\lambda}$
	      satisfies a condition 
	      \begin{align}
	           \rho(0,\gamma 0)<r'-2r.\label{eq_lem_family_of_adm_digit}
	      \end{align}
      \end{enumerate}
      Set $\calQ_{\lambda}=\bigvee\{\gamma\calP\mid\calP\in\frakF,
           \gamma\in\mathcal{D}_{\calP}^{\lambda}\}$ for each $\lambda\in\Lambda$.
       Then the family $\{\calQ_{\lambda}\mid\lambda\in\Lambda\}$
       satisfies the first two conditions of the definition of 
       family of building blocks for $r'$ (Definition \ref{def_family_build_block}).
\end{lem}
\begin{proof}
       \emph{The first condition.}
        Take $\lambda\in\Lambda$ and fix it. Since $\supp\calQ_{\lambda}=
        \overline{
         \bigcup_{\calP\in\frakF,\gamma\in\mathcal{D}_{\calP}^{\lambda}}\supp\gamma\calP}$,
         it is nonempty.
         We have moreover 
              $\supp\gamma\calP\subset B(\gamma 0,r)\subset B(0,r')$
         by (\ref{eq_lem_family_of_adm_digit}), for each $\calP\in\frakF$ and
         $\gamma\in\mathcal{D}_{\calP}^{\lambda}$, and so
         $\supp\calQ_{\lambda}\subset B(0,r')$.

        \emph{The second condition.}
         Take $\lambda,\mu\in\Lambda$ and $\gamma,\eta\in\Gamma$ such that
         $\rho(\gamma 0,\eta 0)>4r'$.
         We show that $\gamma\calQ_{\lambda}$ and $\eta\calQ_{\mu}$ are compatible.
         For each $\calP,\calQ\in\frakF$, $\xi\in\mathcal{D}_{\calP}^{\lambda}$ and
         $\zeta\in\mathcal{D}_{\calQ}^{\mu}$, by (\ref{eq_lem_family_of_adm_digit}),
         we have $\rho(\gamma\xi 0,\eta\zeta 0)>4r$.
         Thus $\gamma\xi\calP$ and $\eta\zeta\calQ$ are compatible and so 
         together with Lemma \ref{lem_admissible_digit},
         the set
         $\Xi_1\cup\Xi_2$ is locally finite and pairwise compatible. Here,
         \begin{align*}
	      \Xi_1=\{\gamma\xi\calP\mid\calP\in\frakF,\xi\in\Gamma_{\calP}^{\lambda}\},
	 \end{align*}
         and
         \begin{align*}
	      \Xi_2=\{\eta\zeta\calQ\mid\calQ\in\frakF,\zeta\in\Gamma_{\calQ}^{\mu}\}.
	 \end{align*}
         By Lemma \ref{lem_family_of_loc_fin_pair_comp_sets} and the fact that
         $\gamma\calQ_{\lambda}=\bigvee\Xi_1$ and $\eta\calQ_{\mu}=\bigvee\Xi_2$,
         we see $\gamma\calQ_{\lambda}$ and $\eta\calQ_{\mu}$ are compatible.
\end{proof}

\begin{rem}
       In Lemma \ref{lem_first_two_aixon_of_fam_build_block}, the third condition
      of the definition of family of building blocks
      (Definition \ref{def_family_build_block})
       is  not always satisfied.
       When we use this lemma in Subsection \ref{subsection_proof_translation_thm},
       we prove the third condition in an ad hoc way.
\end{rem}

\begin{lem}\label{useful_lem_for_building_block}
       Take $r>0$ arbitrarily.
       Let $\mathfrak{F}$ be a family of building blocks
       for $r$.
        Take two admissible digits $(\mathcal{D}_{\calP}^1)_{\calP\in\mathfrak{F}}$
        and $(\mathcal{D}_{\calP}^2)_{\calP\in\mathfrak{F}}$.
        Suppose both $\bigcup_{\calP}\mathcal{D}_{\calP}^1$ and
        $\bigcup_{\calP}\mathcal{D}_{\calP}^2$ are finite.
         Suppose also that 
         \begin{align*}
	  	    \bigvee \{\gamma \calP\mid\calP\in\mathfrak{F}, \gamma\in\mathcal{D}_{\calP}^1\}
                = \bigvee\{\gamma \calP\mid\calP\in\mathfrak{F},\gamma\in\mathcal{D}_{\calP}^2\}.
	 \end{align*}
         Then for any $\calP\in\mathfrak{F}$ and $\gamma\in\mathcal{D}_{\calP}^1$
 there is
        $\eta\in\mathcal{D}_{\calP}^2$ such that $\gamma\calP=\eta\calP$.
\end{lem}
\begin{proof}
                Consider two finite sets
         \begin{align*}
	    F_1= \{\gamma 0\mid\gamma\in\bigcup_{\calP}\mathcal{D}_{\calP}^1\}
	 \end{align*}
          and 
         \begin{align*}
F_2  = \{\gamma 0\mid\gamma\in\bigcup_{\calP}\mathcal{D}_{\calP}^2\}.
          \end{align*}
        For each $x\in F_1$, 
       there are $\calP\in\mathfrak{F}$ and $\gamma\in\Gamma_{\calP}^1$
        such that $x=\gamma 0$. Set $\calP_x^1=\gamma\calP$. This is independent of
        the choice of $\calP$ and $\gamma$. Define $\calP_x^2$ 
       for each $x\in F_2$ in a similar way.
      The claim follows from Lemma \ref{useful_lem_for_last_half_of_main_thm}.
\end{proof}

\subsection{Preliminary Lemmas}
\label{subsection_preliminary_lemma}


Here we prove some technical lemmas which are used in the next subsection.

\begin{lem}\label{assumption_existence_y_G}
        Let $\mathfrak{G}$ be a set of subgroups of $\Gamma_0$ which is at most countable.
        Suppose $\max_{G\in\mathfrak{G}}\card G<\infty$.
        Then for each two numbers $r,s$ such that $r>s>0$,
        there are $\e>0$ and 
        a point $y_G\in B(0,r)^{\circ}\setminus B(0,s)$ for each $G\in\mathfrak{G}$
         such that
        \begin{enumerate}
	 \item if $G\in\mathfrak{G}$ and $\gamma\in G\setminus\{e\}$, then
                $\rho(y_G,\gamma y_G)>\e$, and
        \item if $G\neq H$, then $\rho(0,y_G)\neq\rho(0,y_H)$.
	\end{enumerate}
\end{lem}

To prove Lemma \ref{assumption_existence_y_G},
we prepare the following notation.

\begin{nota}
      For any $A\in\Od$, $r>0$ and $\e\geqq 0$, set
       \begin{align*}
	     S_{A,\e,r}=\{x\in B(0,r)\mid \rho(Ax,x)\leqq\e\}.
       \end{align*}
\end{nota}

We prove two lemmas beforehand to prove Lemma \ref{assumption_existence_y_G}.
\begin{lem}
      If the order of an element $A\in\Od$ is less than an integer $m$, then
      $S_{A,\e,r}\subset S_{A,0,r}+B(0,\frac{m}{2}\e)$.
\end{lem}
\begin{proof}
       Take an element $x\in S_{A,\e,r}$. Let $k$ be the order of $A$.
       Set $y=\frac{1}{k}\sum_{j=0}^{k-1}A^j x$.
      By convexity of $B(0,r)$, $y$ is in $B(0,r)$, and so
      $y\in S_{A,0,r}$. Moreover,
      \begin{align*}
            \rho(x,y)&=\|\frac{1}{k}\sum(A^jx-x)\|\\
                    &\leqq\frac{1}{k}\sum_{j=0}^{k-1}\sum_{i=0}^{j-1}\|A^ix-A^{i+1}x\|\\
                   &\leqq\frac{1}{k}\sum_{j=0}^{k-1}j\e\\
                  &\leqq\frac{m}{2}\e.
      \end{align*}
\end{proof}

\begin{lem}
       Let $m$ be a positive integer and $r$ be a positive real number.
       We have $\lim_{\e\rightarrow 0}\mu(S_{A,\e,r})=0$ uniformly for all $A\in\Od\setminus\{e\}$
 such that
       the order of $A$ is less than $m$.
\end{lem}
\begin{proof}
       For each such $A$ there is a $d-1$ dimensional vector subspace $V_A$ of $\Rd$ such that
       $S_{A,0,r}\subset V_A\cap B(0,r)$, and so 
       $S_{A,\e,r}\subset (V_A\cap B(0,r))+B(0,\frac{m}{2}\e)$.
       For any $d-1$ dimensional vector subspace $V$ of $\Rd$,
       the limit $\lim_{\e\rightarrow 0}\mu((V\cap B(0,r))+B(0,\frac{m}{2}\e))$ converges uniformly
      to $0$.
\end{proof}

\begin{proof}[Proof of Lemma \ref{assumption_existence_y_G}.]
      If $\e$ is small enough, for any $A\in\Od\setminus\{e\}$ of  order at most $m$,
      $m\mu(S_{A,\e,r})<\mu(B(0,r)^{\circ}\setminus B(0,s))$.
      This implies that $B(0,r)^{\circ}\setminus B(0,s)$ is not included in 
      $\bigcup_{A\in G,A\neq e}S_{A,\e,r}$ for any $G\in\mathfrak{G}$.
      To take each $y_G$, we enumerate $\mathfrak{G}$ as $\mathfrak{G}=\{G_1,G_2,\ldots\}$.
       First take 
       $y_{G_1}\in B(0,r)^{\circ}\setminus (B(0,s)\cup\bigcup_{A\in G_1,A\neq e}S_{A,\e,r})$.
       If we have taken $y_{G_1}, y_{G_2},\ldots, y_{G_{n-1}}$, we can take
       $y_{G_n}\in B(0,r)^{\circ}\setminus (B(0,s)\cup\bigcup_{A\in G_n,A\neq e}S_{A,\e,r})$
       such that $\|y_{G_n}\|\neq \|y_{G_j}\|$ for each $j=1,2,\ldots, n-1$.
       In this way, we can take $y_{G_1},y_{G_2},\ldots$ with the desired condition.
\end{proof}

We finish this subsection by proving two lemmas, which we will use to prove Theorem
\ref{translation_thm}.
     For each $j=1,2,\ldots, d$, let $e_j\in\Rd$ be the vector
      of which $i$th component is $0$ for $i\neq j$ and $j$th component
     is $1$.

\begin{lem}\label{assumption_existence_F}
        For any $r>0$ there is a subset $F\subset B(0,r)$ such that
        \begin{itemize}
	 \item  $1<\card F<\infty$, and
         \item $\Sym_{\Gamma}F=\{e\}$.
	\end{itemize}
\end{lem}
\begin{proof}
        Take for each $j=1,2,\ldots, d$ a positive number $r_j>0$.
        Set $F=\{0,r_1e_1,r_2e_2,\ldots ,r_de_d\}$.
        If any two $r_j$'s are different but all close to $1$, then
        $0$ is the only vector in $F$ such that the distances with any other vectors
        are close to $1$.
        Thus if $\gamma\in\Gamma$ and $\gamma F=F$, then $\gamma 0=0$.
        Since  $r_j$'s are all different,$\gamma r_je_j=r_je_j$ for each $j$,
        and since $\{r_1e_1,\ldots ,r_de_d\}$ is a basis for $\Rd$, $\gamma$ must be $e$.        
\end{proof}

\begin{lem}\label{assumption_size_symmetry_group}
      For any $r>0$ and $R>0$ there are $R'>0$ and $C_1>0$ such that, if
      $x\in\Rd$ and
      $D$ is a Delone set of $\Rd$ which is $R$-relatively dense
      and $r$-uniformly
      discrete, then
      \begin{align}
             \card(\Sym_{\Gamma_x}D\cap B(x,R'))<C_1.\label{eq_assum5_translation_thm}
      \end{align}
\end{lem}
\begin{proof}
         Take $R'>0$ large enough so that if $e'_1,e'_2,\ldots, e'_d\in\Rd$ and
         $\|e_j-e'_j\|<\frac{R}{R'-R}$ for each $j$, 
         then $\{e'_1,e'_2,\ldots, e'_d\}$ is linear independent.
         Set $C_1>k!$, where $k$ is an integer such that 
         $k>\frac{\mu(B(0,R'+r))}{\mu(B(0,\frac{r}{2}))}$.

        Take $(R,r)$-Delone set $D$ and $x\in \Rd$ arbitrarily.
        For each $j=1,2,\ldots, d$, there is $x_j\in D\cap B(x+(R'-R)e_j,R)$.
        Then for each $j$ we have $\|\frac{1}{R'-R}(x_j-x)-e_j\|<\frac{R}{R'-R}$
        and so the set of vectors $\{x_j-x\mid j=1,2,\ldots, d\}$ is a basis for $\Rd$.

        If  $\gamma\in\Gamma_x$ and $\gamma(y)=y$ for each $y\in D\cap B(x,R')$, then
        since $\gamma$ fixes $x,x_1,x_2,\ldots, x_d$,
        $\gamma=e$.
        Thus we have an embedding of $\Sym_{\Gamma_x}D\cap B(x,R')$ into
        the permutation group of the set $D\cap B(x,R')$.
        Since for any two distinct $y,z\in D\cap B(x,R')$ we have
        $B(y,r/2)\cap B(z,r/2)=\emptyset$,
        we see $\mu(B(0,r/2))\card D\cap B(x,R')\leqq\mu(B(0,R'+r))$.
        The order of the permutation group is less than $C_1$ which we took above.
        We thus see the inequality (\ref{eq_assum5_translation_thm}).
\end{proof}


\subsection{Proof of translation theorem}
\label{subsection_proof_translation_thm}

\begin{setting}
           \emph{In addition to Setting \ref{setting_of_main_thm}, in this subsection
   we assume $\Sigma$ is a supremum-closed subshift of $\Pi_2$ that
 contains sufficiently many symmetric building blocks, which means that
        for each $r>0$, there is a symmetric building block $\calP_r$ for $r$
        (Definition \ref{def_family_build_block}).}
\end{setting}

Here we prove Theorem \ref{translation_thm}, which answers the second question given in
Introduction.

\begin{thm}\label{translation_thm}
       Let $\calP$ be an abstract pattern in $\Pi_1$ which consists
       of bounded components (Definition \ref{def_bounded_coponents}) and
       is Delone-deriving (Definition \ref{def_Delone_deriving}).
Then there is an abstract pattern $\calS$ in $\Sigma$ such that 
       $\calP\MLD\calS$. Moreover, $\supp\calS$ is relatively dense in $\Rd$.
\end{thm}

\begin{rem}
       This theorem holds if we replace $(\Rd,\Gamma)$ with a pair $(X,\Gamma)$ of a 
       proper metric space $X$
 and a group $\Gamma$ that acts on $X$ transitively as isometries such that
       inequality (\ref{eq_relation_rho_rhoGamma}), Lemma \ref{assumption_existence_y_G},
       Lemma \ref{assumption_existence_F} and Lemma \ref{assumption_size_symmetry_group}
       hold if we replace $2$ on the right-hand side of (\ref{eq_relation_rho_rhoGamma})
        with some positive number
       and $0\in\Rd$ in these assertions with some point in $X$.
\end{rem}

\begin{cor}\label{cor_MLD_with_Delone}
      Under the same assumption as in Theorem \ref{translation_thm} on $\calP$, there is
     a Delone set $D$ in $\Rd$ that is MLD
     with $\calP$.
\end{cor}
\begin{proof}
      If $\Sigma=\UD_r(\Rd)$, the one-point set $\calP=\{0\}$ is a symmetric
      building block and so this $\Sigma$ satisfies the condition in
      Theorem \ref{translation_thm}.
\end{proof}
Note that if $\Gamma$ is bigger than $\Rd$, our ``MLD'' means ``S-MLD'' in \cite{MR1132337}.
\begin{rem}
       In Section \ref{section_application_translation_thm} we give  sufficient
       conditions for a subshift of functions to have sufficiently many
      symmetric building blocks.
      We will be able to apply Theorem \ref{translation_thm} when $\Sigma$ is
      a space of certain functions under a mild condition.
\end{rem}

The strategy of proof can be explained as follows.
We first prove that under a condition we can decompose
an abstract pattern $\calP$ as in Subsection \ref{subsection_decomposition}.
 We  decompose $\calP$ and replace the components
 (the tuple of components $(\calP_{\lambda})_{\lambda}$ as in Subsection
 \ref{subsection_decomposition})
with building blocks ($(\calR_{\lambda})_{\lambda}$ in page \pageref{construction_calR})
in another abstract pattern space. We then assemble such building
blocks $(\calR_{\lambda})_{\lambda}$
in the same way that $\calP$ is constructed from $(\calP_{\lambda})_{\lambda}$
(that is, with respect to the plan),
 to obtain an abstract pattern $\calS$, which is MLD with the original $\calP$.

      We now start the proof of Theorem \ref{translation_thm}.
         Let $\calP\in\Pi_1$ be an abstract pattern that consists of bounded components
          (Definition \ref{def_bounded_coponents}).
         Suppose $\calP$ is Delone-deriving, that is, 
         there is a Delone set $D$ in $\Rd$ such that $\calP\LD D$.
        We will use Proposition \ref{prop_calP_MLD_Gamma_lambda}, and so we first prove
	the following.
\begin{lem}\label{lem_thereis_R_that_decomposes}
         There exists $R_0>0$ such that $(D,R_0)$ decomposes $\calP$ (Definition 
         \ref{def_D_R_decomposes_calP}).
\end{lem}
\begin{proof}
        The set $D$ is Delone so that it is $R_D$-relatively dense for a positive $R_D>0$
        and $r_D$-uniformly discrete for some $r_D>0$.
        For these $R_D$ and $r_D$, there are $R'$ and $C_1$ as in Lemma
        \ref{assumption_size_symmetry_group}.
        The abstract pattern $\calP$ consists of bounded components so that there is
        $R_{\calP}$ as in Definition \ref{def_bounded_coponents}.
        Since $D$ is locally derivable from $\calP$, there is a constant $R_{LD}>0$
        for a point $x_0=y_0=0$ as in 1. of Lemme \ref{lem_local_derivability}.
        Take $R_0>R_D+R_{\calP}+R_{LD}+R'$.

         The first condition of Definition \ref{def_D_R_decomposes_calP} is satisfied
         by the assumption.

        \emph{The Second Condition of Definition \ref{def_D_R_decomposes_calP}.}
         First, we show $\{\calP\sci B(x,R_0)\mid x\in D\}$ is locally finite and 
         pairwise compatible.
        For each $x\in\Rd$ and $r>0$, we have an inclusion
        \begin{align*}
	       \{y\in D\mid B(y,R_0)\cap B(x,r)\neq\emptyset\}\subset D\cap B(x,R_0+r)
	\end{align*}
         and the latter is finite. 
         Hence $\{\calP\sci B(y,R_0)\sci B(x,r)\mid y\in D\}$ is finite, since
         it is a zero element
         except for finitely many $y's$ and by Lemma \ref{uniqueness_zero_element},
         zero element is unique.
        On the other hand, pairwise-compatibility is clear since for each $\calP\sci B(x,R_0)$,
        $\calP$ is a majorant.

        Since $\Pi_1$ is glueable, there is the supremum 
         $\calQ=\bigvee\{\calP\sci B(x,R_0)\mid x\in D\}$.
        On one hand, we see 
        by Lemma \ref{lem_support_supremum} that
        $\supp \calQ=\overline{\bigcup_{x\in D}\supp(\calP\sci B(x,R_0))}
        \subset \supp\calP$;
         on the other hand, if $y\in\supp\calP$, then
         \begin{align*}
	        y\in\supp(\calP\sci B(y,R_{\calP}))\subset\supp(\calP\sci B(x,R_0))
                    \subset\supp\calQ
	 \end{align*}
         for some $x\in D$, and so $\supp\calP\subset\supp\calQ$.
          Therefore $\supp\calP=\supp\calQ$.
          Since $\calP\geqq\calP\sci B(x,R_0)$ for each $x\in D$ and $\calQ$ is the supremum
         of such abstract patterns, we have $\calP\geqq\calQ$.
          Thus $\calP=\calP\sci\supp\calP=\calP\sci\supp\calQ=\calQ$ by the definition of 
           order $\geqq$ (Definition \ref{def_order_pattern_space}).

          \emph{The Third Condition of Definition \ref{def_D_R_decomposes_calP}.}
          For each $x\in D$, take $\gamma\in\Sym_{\Gamma_x}\calP\sci B(x,R_0)$.
          Then
          \begin{align*}
	       (\gamma\calP)\sci B(x,R_0)=\gamma(\calP\sci B(x,R_0))=\calP\sci B(x,R_0),
	  \end{align*}
          and since $\calP\LD D$ with respect to the constant $R_{LD}$,
          we have
          \begin{align*}
	       \gamma(D\cap B(x,R'))=(\gamma D)\cap B(x,R')=D\cap B(x,R').
	  \end{align*}
           This means that $\gamma\in\Sym_{\Gamma_x}D\cap B(x,R')$.
           By definition of $C_1$, 
           $\card\Sym_{\Gamma_x}\calP\sci B(x,R_0)\leqq\card\Sym_{\Gamma_x}D\cap B(x,R')<C_1$.
\end{proof}


By this Lemma \ref{lem_thereis_R_that_decomposes}, there is $R_0>0$ such that
$(D,R_0)$ decomposes $\calP$, and so we can find tuple of components and plan, as follows.

By Lemma \ref{lem_existence_components}, there is a set $\Lambda$ and
a 
 tuple of components $(\calP_{\lambda})_{\lambda\in\Lambda}$. Let $(C_{\lambda})_{\lambda\in\Lambda}$
be the plan for $\calP$ with respect to $(D,R_0,(\calP_{\lambda})_{\lambda})$.
Then we have the following:
\begin{itemize}
          \item      $\Lambda$ is a set which is 
		       at most countable.
          \item Since each $\calP_{\lambda}$ is a copy of an
                abstract pattern of the form $\calP\sci B(x,R_0)$
		($x\in D$)
                by an element $\gamma\in\Gamma$ such 
               that $\gamma x=0$, by Definition \ref{def_D_R_decomposes_calP} we have the following:
                 $G_{\lambda}=\Sym_{\Gamma_{0}}\calP_{\lambda}$ 
		is a finite group, for each $\lambda\in\Lambda$, and 
		$\max_{\lambda}\card G_{\lambda}<\infty$.
	\item For each $\lambda\in\Lambda$, 
	      $C_{\lambda}$ 
	      is a subset of $\Gamma$ such that
               \begin{align}
		    C_{\lambda}G_{\lambda}= C_{\lambda}.
		     \label{eq_ClambdaGlambda=Clambda}
	       \end{align}
         \item  $D$ is a Delone set such that 
		\begin{align}
		  D=\{\gamma 0\mid \lambda\in\Lambda, \gamma\in C_{\lambda}\}.
		       \label{eq_D_and_Clambda} 
		\end{align}
         \item There is $r_0>0$ such that, 
               \begin{align}
	  \text{ if $\lambda,\mu\in\Lambda, \gamma\in C_{\lambda}$,
	       $\eta\in C_{\mu}$ and $\rho(\gamma 0,\eta 0)\leqq 4r_0$, then
		  $\gamma 0=\eta 0$,}\nonumber\\
	      \text{and so $\lambda=\mu$ and $\gamma^{-1}\eta\in G_{\lambda}$.}
		               \label{condition_C_lambda_admissible_digit}	 
	       \end{align} 
\end{itemize} 

By Proposition \ref{prop_calP_MLD_Gamma_lambda}, we have
$\calP\MLD(C_{\lambda})_{\lambda}$. To prove $\calS\MLD\calP$ for some $\calS\in\Sigma$,
we construct an abstract pattern $\calS$ in $\Sigma$ such that 
$\calS\MLD (C_{\lambda})_{\lambda}$.
It consists of three steps.

\noindent\underline{Step 1: construction of $\calE$.}

As described in the beginning of this section, we will construct a family of
building blocks $(\calR_{\lambda})_{\lambda}$. In order to construct this, we first
construct a building block $\calE$, from which each $\calR_{\lambda}$ is constructed.

By Lemma \ref{assumption_existence_y_G}, there are 
$y_{\lambda}\in B(0,\frac{3}{4}r_0)\setminus B(0,\frac{1}{2}r_0)$ for each 
$\lambda\in\Lambda$ and $r_1\in (0,\frac{1}{8}r_0)$ such that
\begin{itemize}
 \item $\inf\{\rho(\gamma y_{\lambda},y_{\lambda})\mid \lambda\in\Lambda, 
       \gamma\in G_{\lambda}\setminus\{e\}\}>4r_1>0$, and
\item if $\lambda,\mu$ are two distinct elements of $\Lambda$, then we have
       $\rho(0,y_{\lambda})\neq\rho(0,y_{\mu})$.
\end{itemize}

By Lemma \ref{assumption_existence_F}, there are $F\subset B(0,\frac{1}{2}r_1)$
and $r_2\in (0,\frac{1}{4}r_1)$ such that
\begin{itemize}
 \item If $x,y\in F$ and $x\neq y$, then $\rho(x,y)>4r_2$,
\item $\Sym_{\Gamma}F=\{e\}$, and
\item $\infty>\card F>1$.
\end{itemize}

Take $\gamma_x\in\Gamma$, for each $x\in \Rd$, such that $\gamma_x0=x$.

\begin{nota}
       Let $\calP_0$ be a symmetric building block of $\Sigma$ for $r_2$.
        (Its existence is assumed in Setting \ref{setting_of_main_thm}.)
       Set $\calE=\bigvee\{\gamma_x\calP_0\mid x\in F\}$.
\end{nota}

\begin{rem}
        Since points of $F$ are separated by the distance $4r_2$,
        by the definition of building block the set $\{\gamma_x\calP_0\mid x\in F\}$
       is pairwise compatible. Since it is a finite set, it is locally finite.
       Its supremum exists.
\end{rem}         

If $\Sigma$ is the set of all $r$-uniformly discrete sets in $\Rd$, this $\calE$ is
nothing but $F$ itself, which has trivial symmetry.
In general cases we can also prove that $\calE$ has trivial symmetry:

\begin{lem}\label{lem_trivial_symmetry_calE}
       $\Sym_{\Gamma}\calE=\{e\}$.
\end{lem}
\begin{proof}
         Take $\gamma\in\Gamma$ such that $\gamma\calE=\calE$.
         Since $\gamma\calE=\bigvee\{\gamma\gamma_x\calP_0\mid x\in F\}$,
         by Lemma \ref{useful_lem_for_building_block},
         for each $x\in F$ there is $y\in F$ such that 
         $\gamma\gamma_x\calP_0=\gamma_y\calP_0$.
         By the definition of building block (Definition \ref{def_family_build_block}), 
we have $\gamma\gamma_x0=\gamma_y0$
         and $\gamma x=y$.
         This implies that $\gamma F\subset F$ and $\gamma F=F$,
         which implies that $\gamma=e$.
\end{proof}

We use $\{\calP_0,\calE\}$ to construct $\calR_{\lambda}$'s.  To this aim we need to show
the following:
\begin{lem}\label{lem_calP_calE_family_building_block}
        The set $\{\calP_0,\calE\}$ is a family of building blocks of $\Sigma$
        for $r_1$.
\end{lem}
\begin{proof}
        We apply Lemma \ref{lem_first_two_aixon_of_fam_build_block}.
        The sets $\{e\}$ and $\{\gamma_x\mid x\in F\}$ play the role of 
        admissible digits. If $x\in F$, then
        \begin{align*}
	      \rho(\gamma_x 0,0)=\rho(x,0)\leqq\frac{1}{2}r_1<r_1-2r_2,
	\end{align*}
        and so by Lemma \ref{lem_first_two_aixon_of_fam_build_block} the first
        two axioms for family of building blocks are satisfied.

        Since $\calP_0$
 is a building block, we have $\Sym_{\Gamma}\calP_0\subset\Gamma_{0}$.
        Moreover, $\Sym_{\Gamma}\calE=\{e\}\subset\Gamma_{0}$.
        Finally,
        we never have $\gamma\calP_0=\calE$ for any $\gamma\in\Gamma$. If this holds
        we have, by Lemma \ref{useful_lem_for_building_block},
        $\gamma_x\calP_0=\gamma\calP_0$
 for any $x\in F$, and this implies $x=\gamma_x0=\gamma 0$
         for each $x\in F$.
         This contradicts the fact that $\card F>1$.
\end{proof}

\noindent\underline{Step 2: construction of $\calR_{\lambda}$.}
\label{construction_calR}
For each $\lambda\in\Lambda$, set
\begin{align*}
       \calR_{\lambda}=\bigvee\{\calP_0\}\cup\{\gamma\gamma_{y_{\lambda}}\calE\mid\gamma\in G_{\lambda}\}.
\end{align*}

We use $\calR_{\lambda}$'s to construct $\calS$. To this aim we need to show the
following:
\begin{lem}
       The set $\{\calR_{\lambda}\mid\lambda\in\Lambda\}$ is a family of building blocks for $r_0$.
\end{lem}
\begin{proof}
        Since $\gamma 0=0$, we have for each $\gamma\in G_{\lambda}$,
          \begin{align*}
	          \rho(\gamma\gamma_{y_{\lambda}}0,0)=\rho(0,y_{\lambda})>\frac{1}{2}r_0>4r_1,
	  \end{align*}
         and by definition of $y_{\lambda}$'s, for each distinct 
         $\gamma,\eta\in G_{\lambda}$,
         \begin{align*}
	         \rho(\gamma\gamma_{y_{\lambda}}0,\eta\gamma_{y_{\lambda}}0)=
                         \rho(\eta^{-1}\gamma y_{\lambda},y_{\lambda})>4r_1
	 \end{align*}
          we see the pair of $\{e\}$ and $\{\gamma\gamma_{y_{\lambda}}\mid\gamma\in G_{\lambda}\}$
          forms an admissible digit for each $\lambda\in\Lambda$.
          Moreover, by
          \begin{align*}
	          \rho(\gamma\gamma_{y_{\lambda}}0,0)=\rho(0,y_{\lambda})
	                                                \leqq \frac{3}{4}r_0<r_0-2r_1
	  \end{align*}
           we see, by Lemma \ref{lem_first_two_aixon_of_fam_build_block}, 
           the first two axioms
           for the building block are satisfied.

          Suppose $\lambda,\mu
           \in\Lambda$, $\gamma_0\in\Gamma$ and $\gamma_0\calR_{\lambda}=\calR_{\mu}$.
          By Lemma \ref{useful_lem_for_building_block}, we have $\gamma_0\calP_0=\calP_0$
       and so $\gamma_0 0=0$.
          Again by Lemma \ref{useful_lem_for_building_block},
          there is $\gamma\in G_{\mu}$ such that
          $\gamma_0\gamma_{y_{\lambda}}\calE=\gamma\gamma_{y_{\mu}}\calE$, and so
          by Lemma \ref{lem_trivial_symmetry_calE}, 
          $\gamma_0\gamma_{y_{\lambda}}=\gamma\gamma_{y_{\mu}}$.
          This implies that (since $\gamma_0$ and $\gamma$ fix $0$)
            \begin{align*}
	     \rho(0,y_{\lambda})=\rho(0,\gamma_0\gamma_{y_{\lambda}}0)
         =\rho(0,\gamma\gamma_{y_{\mu}}0)
                  =\rho(0,y_{\mu})
	    \end{align*}
            and so $\lambda=\mu$.
\end{proof}

We need the fact that $\calR_{\lambda}$ has the same symmetry group as $\calP_{\lambda}$:
\begin{lem}\label{lem_symmetry_Rlambda_Glambda}
        $\Sym_{\Gamma}\calR_{\lambda}=G_{\lambda}$ for each $\lambda$.
\end{lem}
\begin{proof}
       Take $\gamma_0\in\Sym_{\Gamma}\calR_{\lambda}$.
       By Lemma \ref{useful_lem_for_building_block}, there is $\gamma\in G_{\lambda}$
       such that $\gamma_0\gamma_{y_{\lambda}}\calE=\gamma\gamma_{y_{\lambda}}\calE$
       and so by Lemma \ref{lem_trivial_symmetry_calE} we have $\gamma_0=\gamma\in G_{\lambda}$.

       On the other hand, if $\gamma_0\in G_{\lambda}$, then
       \begin{align*}
       	   \gamma_0\calR_{\lambda}=&\bigvee\{\gamma_0\calP_0\}\cup
             \{\gamma_0\gamma\gamma_{y_{\lambda}}\calE\mid\gamma\in G_{\lambda}\}\\
              =&\bigvee\{\calP_0\}\cup
              \{\gamma\gamma_{y_{\lambda}}\calE\mid\gamma\in G_{\lambda}\}\\
             =&\calR_{\lambda},
       \end{align*}
       since $\calP_0$ is a symmetric building block.
\end{proof}

\noindent\underline{Step 3: Construction of $\calS$ and its property.}

Here we construct $\calS\in\Sigma$ by using $\calR_{\lambda}$'s in the exactly
same way as $\calP$ is constructed from $\calP_{\lambda}$'s, that is,
with respect to the plan $(C_{\lambda})_{\lambda}$.
We show $\calS$ and $(C_{\lambda})_{\lambda}$ are MLD (Theorem \ref{thm_calS_MLD_Clambda})
by using Proposition \ref{prop_calP_MLD_Gamma_lambda}, and this shows that
$\calP$ and $\calS$ are MLD since $\calP$ and $(C_{\lambda})_{\lambda}$ are also MLD
by again using Proposition \ref{prop_calP_MLD_Gamma_lambda}.

Define
\begin{align*}
        \calS=\bigvee\{\gamma\calR_{\lambda}\mid\lambda\in\Lambda, \gamma\in C_{\lambda}\}.
\end{align*}
by (\ref{condition_C_lambda_admissible_digit}), $(C_{\lambda})_{\lambda}$ is an admissible digit 
for $(\calR_{\lambda})_{\lambda\in\Lambda}$ and so $\calS$ is well-defined.
We first prove $(D,r_0)$ decomposes $\calS$ (Lemma \ref{lem_D_r_0_dec_calS}), and then
prove $(C_{\lambda})_{\lambda\in\Lambda}$ is the plan
(Lemma \ref{lem_Clambda_plan_for_calS}). By these we can use Proposition \ref{prop_calP_MLD_Gamma_lambda}.

In the following two lemmas we check the conditions in Definition \ref{def_D_R_decomposes_calP}.
\begin{lem}\label{lem_calS_LD_D}
      $\calS\LD D$.
\end{lem}
\begin{proof}
      Let $R$ be an arbitrary positive real number.
      Set $L=R+3r_0$.
       Assume $\gamma,\eta\in\Gamma$ and
       \begin{align}
	    (\gamma\calS)\sci B(0,L)=(\eta\calS)\sci B(0,L).\label{eq1_calS_LD_D}
       \end{align}
       Set $\Xi=\{\xi\calR_{\lambda}\mid\lambda\in\Lambda, \xi\in C_{\lambda}\}$, then by
       (\ref{eq1_calS_LD_D}) and Lemma \ref{lem_sup_action_commute}, we see
       \begin{align*}
	     \bigvee(\gamma\Xi\sci B(0,L))=\bigvee(\eta\Xi\sci B(0,L)).
       \end{align*}
       Consider the following two finite sets:
       \begin{align*}
	F_1=\{\gamma\xi 0\mid\lambda\in\Lambda,\xi\in C_{\lambda},
             \gamma\xi\calR_{\lambda}\sci B(0,L)\neq 0\}
       \end{align*}
        and
        \begin{align*}
	 F_2=\{\eta\zeta 0\mid\lambda\in\Lambda,\zeta\in C_{\lambda},
             \eta\zeta\calR_{\lambda}\sci B(0,L)\neq 0\}.
	\end{align*}
         For each $x=\gamma\xi 0\in F_1$, we consider an abstract pattern 
         $\calP_x^1=\gamma\xi\calR_{\lambda}\sci B(0,L)$.
          This is included in $B(\gamma\xi 0,r_0)$.
         For $F_2$ we define $\calP_x^2$'s in a similar way.
         We can apply Lemma \ref{useful_lem_for_last_half_of_main_thm} and obtain
        the following:
        if $\lambda\in\Lambda$, $\xi\in C_{\lambda}$ and $\gamma\xi\calR_{\lambda}
        \sci B(0,L)\neq 0$,
        there is $\mu\in\Lambda$ and $\zeta\in C_{\mu}$ such that
        \begin{align*}
	       (\gamma\xi\calR_{\lambda})\sci B(0,L)=(\eta\zeta\calR_{\mu})\sci B(0,L).
	\end{align*}

        Now we prove $(\gamma D)\cap B(0,R)\subset(\eta D)\cap B(0,R)$.
        Take an element $\gamma\xi 0$ from the left-hand side set, where $\xi\in C_{\lambda}$ for some
        $\lambda$ and $\gamma\xi 0\in B(0,R)$.
        Then $\supp\gamma\xi R_{\lambda}\subset B(0,R+r_0)$. As in the previous paragraph,
        there are $\mu\in\Lambda$ and $\zeta\in C_{\mu}$ such that
       $\gamma\xi \calR_{\lambda}=(\eta\zeta\calR_{\mu})\sci B(0,L)$. The support of 
        this abstract pattern is included in $B(0,R+r_0)$ and the support of $\eta\zeta\calR_{\mu}$
        has a diameter less than $2r_0$; we have
        $\supp(\eta\zeta\calR_{\mu})\subset B(0,L)$ and so
        $\gamma\xi\calR_{\lambda}=\eta\zeta\calR_{\mu}$.
        Since $(\calR_{\lambda})_{\lambda}$ is a family of building blocks, we see $\lambda=\mu$ and
        $\gamma\xi 0=\eta\zeta 0\in\eta D$.
        We have proved $(\gamma D)\cap B(0,R)\subset(\eta D)\cap B(0,R)$ and by symmetry
       the reverse inclusion is true. 
\end{proof}

\begin{lem}\label{lem_calS_B(gamma0_r0)}
      For each $\lambda\in\Lambda$ and $\gamma\in C_{\lambda}$, we have
       \begin{align*}
	    \calS\sci B(\gamma 0,r_0)=\gamma\calR_{\lambda}.
       \end{align*}     
\end{lem}
\begin{proof}
       If $\mu\in\Lambda$, $\eta\in C_{\mu}$ and $\eta\calR_{\mu}\sci B(\gamma 0,r_0)\neq 0$, then
       $\rho(\gamma 0,\eta 0)\leqq 2r_0$ and so by (\ref{condition_C_lambda_admissible_digit}),
        $\lambda=\mu$ and $\gamma\calR_{\lambda}=\eta\calR_{\mu}$. Hence
        \begin{align*}
	       \calS\sci B(\gamma 0,r_0)&=\bigvee\{\eta\calR_{\mu}\sci B(\gamma 0,r_0)\mid
                     \mu\in\Lambda, \eta\in C_{\mu}\}\\
                       &=\bigvee\{\gamma\calR_{\lambda}\sci B(\gamma 0 ,r_0)\}\\
                       &=\gamma\calR_{\lambda}.
	\end{align*}   
\end{proof}

\begin{lem}\label{lem_D_r_0_dec_calS}
       The pair $(D,r_0)$ decomposes $\calS$.
\end{lem}
\begin{proof}
       Clear by the definition of $\calS$,
        Lemma \ref{lem_calS_B(gamma0_r0)} and Lemma \ref{lem_calS_LD_D}
        and Lemma \ref{lem_symmetry_Rlambda_Glambda}.
\end{proof}

\begin{lem}\label{lem_Clambda_plan_for_calS}
       $(\calR_{\lambda})_{\lambda}$ is a tuple of components for $\calS$ 
       with respect to
       $(D,r_0)$ and $(C_{\lambda})$ is the plan for $\calS$ with respect to 
       $(D,r_0,(\calR_{\lambda})_{\lambda\in\Lambda})$.
\end{lem}
\begin{proof}
       Take $x\in D$ arbitrarily. By (\ref{eq_D_and_Clambda}), there are
       $\lambda\in\Lambda$ and $\gamma\in C_{\lambda}$ such that
       $x=\gamma 0$ and by Lemma \ref{lem_calS_B(gamma0_r0)},
       $\calS\sci B(x,r_0)=\calS\sci B(\gamma 0,r_0)=\gamma\calR_{\lambda}.$
       The uniqueness of such $\lambda$ is clear since $(\calR_{\mu})_{\mu\in\Lambda}$
       is a family of building blocks. We have shown that $(\calR_{\mu})$ is 
       a tuple of components for $\calP_0$ with respect to $(D,r_0)$.

      Next we show that $(C_{\lambda})_{\lambda}$ is the plan.
      If $\mu\in\Lambda$ and $\gamma\in C_{\mu}$, then
      $\gamma 0\in D$ by (\ref{eq_D_and_Clambda}) and
      $\calS\sci B(\gamma 0,r_0)=\gamma\calR_{\mu}$ by Lemma
       \ref{lem_calS_B(gamma0_r0)}, and so 
       $\gamma\in P_{\mu}(\calS,D,r_0,(\calR_{\lambda}))$.
       Conversely, if $\gamma\in P_{\mu}(\calS,D,r_0,(\calR_{\lambda}))$, then
       $\gamma 0\in D$ and $\gamma\calR_{\mu}=\calS\sci B(\gamma 0,r_0)$.
       By (\ref{eq_D_and_Clambda}), there is $\nu\in\Lambda$ and $\eta\in C_{\nu}$
       such that $\gamma 0=\eta 0$, and so by Lemma \ref{lem_calS_B(gamma0_r0)},
       $\eta \calR_{\nu}=\calS\sci B(\eta 0,r_0)$. 
       This implies that $\gamma\calR_{\mu}=\eta\calR_{\nu}$, and so
       $\mu=\nu$ and $\eta^{-1}\gamma\in\Sym_{\Gamma}\calR_{\mu}=G_{\mu}$.
      By (\ref{eq_ClambdaGlambda=Clambda}), we see $\gamma=\eta\eta^{-1}\gamma\in
        C_{\mu}G_{\mu}=C_{\mu}$.
       We have proved $C_{\mu}=P_{\mu}(\calS,D,r_0,(\calR_{\lambda}))$ for 
      any $\mu\in\Lambda$.
\end{proof}

\begin{thm}\label{thm_calS_MLD_Clambda}
       $\calS\MLD (C_{\lambda})_{\lambda}$.
\end{thm}
\begin{proof}
      Clear from Lemma \ref{lem_Clambda_plan_for_calS} and Proposition \ref{prop_calP_MLD_Gamma_lambda}.
\end{proof}

\begin{cor}
     $\calP\MLD\calS$.
\end{cor}
\begin{proof}
      By  Proposition 
       \ref{prop_calP_MLD_Gamma_lambda},
       $\calP\MLD (C_{\lambda})$
      because $(C_{\lambda})$ is a plan for $\calP$.
       Combined with Theorem \ref{thm_calS_MLD_Clambda} we have $\calP\MLD\calS$.
\end{proof}

\begin{lem}
      $\supp\calS$ is relatively dense.
\end{lem}
\begin{proof}
      For any $x\in\Rd$ there is $y\in D$ near $x$.
      By (\ref{eq_D_and_Clambda}), there are $\lambda\in\Lambda$ and 
      $\gamma\in C_{\lambda}$ such that
      $y=\gamma 0$. Since $\supp\gamma\calR_{\lambda}\subset B(y,r_0)$,
      any point in $\supp\gamma\calR_{\lambda}$, which is a point in
      $\supp\calS$, is near $x$.
\end{proof}

This lemma completes the proof of Theorem \ref{translation_thm}.

\section{A study of abstract patterns via arrows}
\label{section_application_translation_thm}
In this section we study the theory of pattern-equivariant functions in terms of
local derivability, by studying the graph with abstract patterns as vertices and
local derivability as edges.
We prove that the space of all pattern-equivariant functions
contains all of the information of the original abstract pattern up to MLD
(Theorem \ref{thm1_pat-equi_remembers} and Theorem \ref{thm_pat_equiv_ft_remembers_P}).

\subsection{The role of maximal elements}
We start with a definition in an abstract setting:
\begin{defi}\label{defi_Sigma_calP}
         Let $\Pi$ be an abstract pattern space over $(X,\Gamma)$ and $\Pi'$  a
          abstract pattern space over $(Y,\Gamma)$, where $\Gamma$ is a group
          which acts on metric spaces $X$ and $Y$ respectively as isometries.
         Let $\Sigma$ be a subshift of $\Pi'$.
          For each $\calP\in\Pi$, we set
          \begin{align*}
	         \Sigma_{\calP}=\{\calQ\in\Sigma\mid\calP\LD\calQ\}.
	  \end{align*}
\end{defi}

In order to study the relations between $\calP$ and
$\Sigma_{\calP}$, the maximal elements of $\Sigma_{\calP}$, that is,
the elements $\calQ\in\Sigma$ such that $\calP\MLD\calQ$, are useful, as the following
lemma shows:

\begin{lem}\label{lem_max_element_sigma_coincide}
   Let $\Pi_j$ be an abstract
 pattern space over $(X_j,\Gamma)$, for each $j=1,2,3$,
     where $X_j$ is a  metric space on which a group $\Gamma$ acts as isometries.
   Suppose a subshift $\Sigma$ of $\Pi_3$ satisfies the following condition:
    \begin{itemize}
     \item for each $\calP_1\in\Pi_1$, there is $\calP_1'\in\Sigma$ such that
	     $\calP_1\MLD\calP_1'$, and
    \item for each $\calP_2\in\Pi_2$, there is $\calP_2'\in\Sigma$ such that
	  $\calP_2\MLD\calP_2'$.
    \end{itemize}
      Then for each $\calP_1\in\Pi_1$ and $\calP_2\in\Pi_2$,
 we have $\calP_1\MLD\calP_2$ if and only if
      $\Sigma_{\calP_1}=\Sigma_{\calP_2}$.
\end{lem}

\begin{proof}
        Take $\calP_1\in\Pi_1$ and $\calP_2\in\Pi_2$.
         There are $\calP_1',\calP_2'\in\Sigma$ as in the condition above.
         If $\calP_1\LD\calP_2$, then for each $\calQ\in\Sigma_{\calP_2}$, we have
        $\calP_1\LD\calQ$ by the transitivity of local derivability, and so
        $\calQ\in\Sigma_{\calP_1}$.
         Thus if $\calP_1\MLD\calP_2$, then $\Sigma_{\calP_1}=\Sigma_{\calP_2}$.
         On the other hand, if $\Sigma_{\calP_1}=\Sigma_{\calP_2}$, then
         $\calP_1'\in\Sigma_{\calP_2}$ and so $\calP_2\LD\calP_1'\LD\calP_1$.
         By transitivity, we have $\calP_2\LD\calP_1$. Similarly $\calP_1\LD\calP_2$ and
        so $\calP_1\MLD\calP_2$.
\end{proof}

\subsection{Pattern-equivariant functions without $\Od$-actions and their generalizations}
Next, we move on to the theory of pattern equivariant functions.
We will show that for certain $\Sigma$ consisting of functions, $\Sigma_{\calP}$ is
the space of pattern equivariant functions.
First, we recall the definition of pattern equivariant functions.
Kellendonk \cite{MR1985494} defined pattern-equivariant functions for
tilings or Delone sets in order to study the cohomology of the tiling spaces.
We recall the definitions here.


\begin{defi}[\cite{MR1985494}]\label{def_pat_equiv_Kellendonk}
      Let $D$ be a subset of $\Rd$ and $Y$ be a set. 
      A function $f\colon\Rd\rightarrow Y$ is said to be (strongly)
       \emph{$D$-equivariant}
      if there is $R>0$ such that
      $x,y\in\Rd$ and $(D-x)\cap B(0,R)=(D-y)\cap B(0,R)$ imply
      $f(x)=f(y)$.
\end{defi}

It is easy to show that this definition can be rephrased in terms of local
derivability:
\begin{lem}\label{lem_pat_equiv_local_der_Kellendonk}
      Let $D$ be a subset of $\Rd$, $Y$ be a set and $y_0\in Y$.
       Then for any $f\in\Map(\Rd,Y)$, $f$ is $D$-equivariant if and only if
       $D\LD f$.
       Here we regard $D$ as an element of $2^X$
       (Example \ref{2X_as_Gamma_pattern_sp}), which is an
       abstract pattern space over $(\Rd,\Rd)$, and $f$ as an element of 
       $\Map(\Rd,Y,y_0)$ (Example \ref{ex_map_rho}), 
       which is an abstract pattern space over $(\Rd,\Rd)$.
\end{lem}

We generalize the definition of pattern-equivariant function as follows:
\begin{defi}
      Let $X$ be a (proper) metric space and $\Gamma$ be a group which acts on $X$ as
      isometries. Let $\calP$ be an abstract pattern in an abstract pattern space over
       $(X,\Gamma)$. A function $f$ in an abstract pattern space
       $\Map(X,Y,y_0)$ over $(X,\Gamma)$, where $Y$ is a set and $y_0\in Y$
        (Example \ref{ex_map_rho}), is said to be
       $\calP$-equivariant if $\calP\LD f$.
\end{defi}

For a subset $\Sigma$ of $\Map(X,Y,y_0)$, the set $\Sigma_{\calP}$
(Definition \ref{defi_Sigma_calP}) is the set of all $\calP$-equivariant functions in
$\Sigma$.

The following theorem is now easy to prove:
\begin{thm}\label{thm1_pat-equi_remembers}
      Let $D_1$ and $D_2$ be  uniformly discrete subsets
      of a (proper) metric space $X$ on which
     a group  $\Gamma$ acts transitively
       as isometries. Let 
      $\Sigma=C(X,\mathbb{C})$
     be a subshift (Example \ref{ex_subshift_conti_maps}), consisting of continuous
     functions,
    of the pattern space $\Map(X,\mathbb{C},0)$ over $(X,\Gamma)$
      (Example \ref{ex_map_rho}). Then
     $D_1\MLD D_2$ if and only if $\Sigma_{D_1}=\Sigma_{D_2}$ (that is, the space of
      continuous pattern-equivariant functions coincide.)
\end{thm}
\begin{proof}
       By Lemma \ref{lem_max_element_sigma_coincide},
      it suffices to show that for each uniformly discrete $D$ in $X$, there is
    $f\in\Sigma$ such that $D\MLD f$.
     There is $r>0$ such that if $x,y\in D$ and $x\neq y$, then
      $\rho(x,y)>2r$. Take $x_0\in X$ and a continuous function $\varphi$ such that
      \begin{enumerate}
       \item $\supp\varphi\subset B(x_0,r)$,
	\item $\varphi(x)=1\iff x=x_0$, and
	 \item for each $\gamma\in\Gamma_{x_0}$ and $x\in X$, we have
	       $\varphi(\gamma x)=\varphi(x)$.
      \end{enumerate}
       Define a continuous function $f\colon X\rightarrow \mathbb{C}$ as follows.
         Take an $x\in X$ and we should determine the value at $x$. If there is
       $\gamma\in\Gamma$ such that $\gamma x_0\in D$ and $\rho(\gamma x_0,x)<r$, then
       put $f(x)=\varphi(\gamma^{-1}x)$. Otherwise set $f(x)=0$. In other words,
      we put the copies of $\varphi$ on each point of $D$. It is easy to show that
      $f$ is continuous and $D\MLD f$.
\end{proof}

\subsection{Pattern-equivariant functions with $\Od$-actions}

 Rand \cite{de2006pattern} generalized the definition of pattern-equivariant functions
 in another way to
incorporate rotations and flips in the 2-dimensional cases.
Recall the notation $\calT\sqcap C$ for a tiling $\calT$ and a closed $C\subset \Rd$
defined in Definition \ref{def_sqcap}.

\begin{defi}[\cite{de2006pattern}]\label{def_pat_equiv_Rand}
      Let $\calT$ be a tiling of $\Rd$, $\Gamma$ a closed subgroup
     of $\Ed$ that contains $\Rd$,
      $G$ an abelian group and $\phi\colon\Gamma_0\rightarrow \Aut(G)$
      a group homomorphism. Here, $\Aut(G)$ is the group of automorphisms 
      of $G$.
     We say a function $f\colon\Rd\rightarrow G$ is \emph{$\calT$-equivariant}
      with representation $\phi$, or is $\phi$-invariant,
     if there is $R>0$ such that
     $x,x'\in\Rd$, $\gamma\in\Gamma_0$ and
     \begin{align*}
           (\calT\sqcap B(x',R))-x'=\gamma(\calT\sqcap B(x,R)-x)
     \end{align*}
    imply $f(x')=\phi(\gamma)(f(x))$.
\end{defi}

In this case, we can also capture pattern-equivariant functions in terms of
local derivability.
For what follows let $\pi\colon\Gamma\ni(a,A)\mapsto A\in\Gamma_0$
      be the projection.

\begin{lem}\label{lem_pat_equiv_local_der_Rand}
       Let $\calT$ be a tiling which consists of bounded components.
       Let $\Gamma$ be a closed subgroup of $\Ed$ that contains $\Rd$,
        $G$ an abelian group and $\phi\colon\Gamma_0\rightarrow\Aut(G)$
       a group homomorphism. 
       Then for any $f\in\Map(\Rd,G)$,
        $f$ is $\calT$-equivariant with representation $\phi$ if and only if
      $\calT\LD f$.
       Here $\calT$ is regarded as an element of $\Patch(\Rd)$
         (Example \ref{ex_patch_Gamma_pattern_sp}),
        which is a pattern 
      space over $(\Rd,\Gamma)$, and $f$ is regarded as an element of 
       $\Map_{\phi\circ\pi}(\Rd,G,e)$ (Example \ref{ex_map_rho}), 
       which is an abstract pattern space over $(\Rd,\Gamma)$.
\end{lem}

For Rand's definition,
it may be that there is no maximal pattern-equivariant functions,
but Theorem \ref{translation_thm}
gives us a sufficient condition for $\calP$ and $\Sigma$ to admit
maximal elements.

By using Theorem \ref{translation_thm} we obtain a result similar to
Theorem \ref{thm1_pat-equi_remembers}, which says the space of pattern-equivariant
functions has all the information of the original object up to MLD.
Here is the setting for the rest of this section:
\emph{let $\Gamma$ be a closed subgroup of $\Ed$ that contains $\Rd$.
       Take a group homomorphism $\phi\colon\Gamma_0\rightarrow GL_m(\mathbb{C})$.
Let $C^{\infty}_{\phi\circ\pi}(\Rd,\mathbb{C}^m,0)$ be the subshift of
$\Map_{\phi\circ\pi}(\Rd,\mathbb{C}^m,0)$ consisting of all smooth elements of
$\Map_{\phi\circ\pi}(\Rd,\mathbb{C}^m,0)$.
(We say a map $f\colon\Rd\rightarrow\mathbb{C}^m$ is smooth if
$\langle f(\cdot),v\rangle$ is smooth for any $v\in\mathbb{C}^m$, where
$\langle\cdot,\cdot\rangle$ is the standard inner product.)}
In order to use Lemma \ref{lem_max_element_sigma_coincide}
to $\Sigma=C^{\infty}_{\phi\circ\pi}(\Rd,G,0)$, we need to show $\Sigma$ 
admits sufficiently many symmetric
building blocks.
We show in two cases there are sufficiently many building blocks
(Lemma \ref{lem_existence_symmetric_building_block_in_Cinfty} and
Lemma \ref{lem2_existence_many_building_blocks}.)

\begin{lem}\label{lem_existence_symmetric_building_block_in_Cinfty}
       Suppose there is $v\in\mathbb{C}^m\setminus\{0\}$ such that $\phi(\gamma) v=v$
       for each $\gamma\in\Gamma_0$.
       Then $C^{\infty}_{\phi\circ\pi}(\Rd,\mathbb{C}^m,0)$ has sufficiently many
        symmetric building blocks: in other words,
        for any $r>0$ there is a symmetric building block
        $g_r$ for $(0,r)$.
\end{lem}
\begin{proof}
     For each $r>0$, set
     \begin{align*}
         f_r(x)=\begin{cases}
		 \exp(-\frac{1}{r^2-\|x\|^2}) &\text{if $\|x\|<r$}\\
                  0 &\text{otherwise}
		\end{cases}     
     \end{align*}
       for each $x\in\Rd$. Then $f_r$ is a smooth real-valued function on $\Rd$.
        Set $g_r(x)=f_r(x)v$.Then $\emptyset\neq\supp g_r\subset B(0,r)$.
        Moreover, if $\gamma,\eta\in\Gamma$ and $\rho(\gamma 0,\eta 0)>4r$, then
        $\gamma g_r$ and $\eta g_r$ are compatible since 
        \begin{align*}
	       g(x)=
                \begin{cases}
		     \gamma g_r(x)&\text{if $x\in B(\gamma 0,r)$}\\
                     \eta g_r(x) &\text{if $x\in B(\eta 0,r)$}\\
                      0         &\text{otherwise}
		\end{cases}
	\end{align*}
        is a majorant.
        Finally $\Sym_{\Gamma}g_r=\Gamma_0$.
\end{proof}

\begin{lem}\label{lem2_existence_many_building_blocks}
        Suppose $\Gamma_0$ is finite.
       Then $C^{\infty}_{\phi\circ\pi}(\Rd,\mathbb{C}^m,0)$ has sufficiently many building
       blocks.
\end{lem}
\begin{proof}
      For any $r>0$, take $x\in\Rd$ and $r_1\in(0,r/4)$ such that $\|x\|<r/2$ and
       if $A\in\Gamma_0$ and $A\neq I$, then $\|Ax-x\|>4r_1$.
       Take $v\in\mathbb{C}^m$ and set $f(x)=f_{r_1}(x)v$
      (we defined $f_{r_1}$ in the proof of Lemma
 \ref{lem_existence_symmetric_building_block_in_Cinfty}.)
       Set $h=\bigvee\{(A,Ax)f\mid A\in\Gamma_0\}$.
 Then $h$ is a symmetric building block.
\end{proof}


\begin{thm}\label{thm_pat_equiv_ft_remembers_P}
       Assume the same assumption as in Lemma 
       \ref{lem_existence_symmetric_building_block_in_Cinfty} or in
       Lemma \ref{lem2_existence_many_building_blocks}.
       Let $\Pi$ and $\Pi'$ be glueable abstract pattern spaces over $(\Rd,\Gamma)$ and take
       $\calP$ and $\calP'$ from $\Pi$ and $\Pi'$ respectively.
       Assume $\calP$ and $\calP'$ are both Delone-deriving
       and consist of bounded components.
      Set $\Sigma=C^{\infty}_{\phi\circ\pi}(\Rd,\mathbb{C}^m,0)$.
       Then $\calP\MLD\calQ$ if and only if $\Sigma_{\calP}=\Sigma_{\calQ}$.
\end{thm}
\begin{proof}
         By Theorem \ref{translation_thm}, for each $\calP\in\Pi\cup\Pi'$ there is
         $f\in\Sigma$ such that $\calP\MLD f$. The claim follows from Lemma
 \ref{lem_max_element_sigma_coincide}.
\end{proof}

Theorem \ref{thm1_pat-equi_remembers} and Theorem \ref{thm_pat_equiv_ft_remembers_P}
shows that, in many cases,  in order to study abstract patterns, it suffices to
study the space $\Sigma_{\calP}$ 
of certain pattern-equivariant functions.
We may regard the space $\Sigma_{\calP}$ as the space of functions
that reflect the structure of $\calP$.
Sometimes in mathematics the set of functions that reflect the structure of
an object remembers the original object.
For example, consider a locally compact abelian group and its dual, or
a smooth manifold $M$ and its space $C^{\infty}(M)$ of smooth functions.
Theorem \ref{thm1_pat-equi_remembers} and
Theorem \ref{thm_pat_equiv_ft_remembers_P} is similar to such phenomena.

\section*{Acknowledgment }
I thank Takeshi Katsura for giving me comments for the draft of this article.
I also thank the referees for giving me advice on various points in the draft,
which significantly improved the article.

\bibliographystyle{amsplain}
\bibliography{tiling}
\end{document}